\documentclass[12pt]{amsart}
\usepackage{amsmath,amssymb,amsthm,color, euscript, enumerate}
\usepackage{dsfont}
\usepackage{longtable}
\usepackage{tikz}
\usepackage{hyperref}

\DeclareMathOperator{\qdim}{{\rm qdim}}

\newcommand{\1}{ \mathbf{1}}

\newcommand{\maru}[1]{{\ooalign{\hfil#1\/\hfil\crcr
\raise.167ex\hbox{\mathhexbox20D}}}}

\newcommand{\ruby}[2]{%
 \leavevmode
 \setbox0=\hbox{#1}%
 \setbox1=\hbox{\tiny #2}%
 \ifdim\wd0>\wd1 \dimen0=\wd0 \end{lemma}se \dimen0=\wd1 \fi
 \hbox{%
   \kanjiskip=0pt plus 2fil
   \xkanjiskip=0pt plus 2fil
   \vbox{%
     \hbox to \dimen0{%
       \tiny \hfil#2\hfil}%
     \nointerlineskip
     \hbox to \dimen0{\mathstrut\hfil#1\hfil}}}}

\DeclareMathOperator*{\fusion}{\boxtimes}

\newcommand{\Z}{\mathbb{Z}}
\newcommand{\C}{\mathbb{C}}
\newcommand{\R}{\mathbb{R}}

\newcommand{\Q}{\mathbb{Q}}

\newcommand{\g}{\mathfrak{g}}

\newcommand{\h}{\mathfrak{h}}

\newcommand{\aut}{\mathrm{Aut}\,}
\newcommand{\Aut}{\mathrm{Aut}\,}

\newcommand{\wt}{\mathrm{wt}}

\newcommand{\p}{\varphi}

\makeatletter \@addtoreset{equation}{section}

\theoremstyle{plain}
\newtheorem{main}{Main Theorem}
\newtheorem{theorem}{Theorem}[section]
\newtheorem{proposition}[theorem]{Proposition}
\newtheorem{lemma}[theorem]{Lemma}
\newtheorem{corollary}[theorem]{Corollary}

\theoremstyle{definition}
\newtheorem{definition}[theorem]{Definition}
\newtheorem{conjecture}[theorem]{Conjecture}

\theoremstyle{remark}
\newtheorem{remark}[theorem]{Remark}
\numberwithin{equation}{section}

\title[Orbifold VOA having group like fusions]{ Cyclic orbifolds of lattice vertex operator algebras having group like fusions}

\author{Ching Hung Lam} %
  \address {Institute of Mathematics, Academia Sinica, Taipei 10617, Taiwan} 
  \email{chlam@math.sinica.edu.tw}

\date{}
\thanks{C.\,H. Lam was partially supported by a research grant AS-IA-107-M02 of Academia Sinica  and MoST grant 104-2115-M-001-004-MY3 of Taiwan}

\newcommand{\sfr}[2]{\leavevmode\kern-.1em
  \raise.5ex\hbox{\the\scriptfont0 #1}\kern-.1em
  /\kern-.15em\lower.25ex\hbox{\the\scriptfont0 #2}}

\begin{document}

\begin{abstract}
Let $L$ be an even (positive definite) lattice and $g\in O(L)$. 
In this article, we prove that the orbifold vertex operator algebra $V_{L}^{\hat{g}}$ has group-like fusion if and only if $g$ acts trivially on the discriminant group $\mathcal{D}(L)=L^*/L$ (or equivalently $(1-g)L^*<L$). We also determine their fusion rings and the corresponding quadratic space structures when $g$ is fixed point free on $L$. By applying our method to some coinvariant sublattices of the Leech lattice $\Lambda$, we prove a conjecture proposed by G. H\"ohn. In addition, we also discuss a construction of certain holomorphic vertex operator algebras of central charge $24$ using the the orbifold vertex operator algebra $V_{\Lambda_g}^{\hat{g}}$.
\end{abstract}
\maketitle


\section{Introduction}\label{intro}
 
The classification of holomorphic vertex operator algebras (VOA) of central charge $24$ is one of the important problems in the theory of vertex operator algebra. In 1993, Schellekens \cite{Sc93} obtained a partial classification by determining the possible Lie algebra structures for the weight one subspace of holomorphic VOA of central charge $24$. It is also believed that the VOA structure of a holomorphic VOA of central charge $24$ is uniquely determined by its weight one Lie algebra. Recently, there has been much progress towards the classification; Schellekens' list was verified mathematically  in \cite{EMS}. Moreover, it has been verified that all $71$ Lie algebras in Schellekens' list can be realized as weight one Lie algebras of some holomorphic VOAs of central charge $24$ \cite{EMS,FLM,Lam,LS,LS3,LS4,LLin,SS}. The uniqueness conjecture was also verified for all the cases with nontrivial weight one subspace \cite{DM,EMS2,KLL,LLin,LS2,LS5,LS6,LS7}.  The main technique is usually referred as to ``Orbifold construction" but the proofs often involved  case by case analysis and  a lot of computer calculations. A uniform approach still seems to be missing. 

Very recently, G. H\"ohn\,\cite{Ho2} has proposed a different idea for studying  the list of Schellekens using automorphisms of Niemeier lattices and Leech lattice. Along with other results, he suggested a more uniform construction of all $71$ holomorphic VOAs in Schellekens' list using the Leech lattice.  H\"ohn's idea \cite{Ho2} may be viewed as  a generalization of the theory of Cartan subalgebras to VOA. The main idea  is to study the commutant subVOA of the subVOA generated by a Cartan subalgebra of $V_1$ and try to construct a holomorphic VOA using certain simple current extensions of  lattice VOAs and some orbifold subVOAs in the Leech lattice VOA. In this article, we will  study his approach and try to elucidate his idea. We will also realize his proposed construction for a special case.   

 Now let us explain his ideas in detail.  
Let $\mathfrak{g}$ be a Lie algebra in Schellekens' list and let $V$ be  a strongly regular holomorphic VOA of central charge $24$ such that $V_1\cong \mathfrak{g}$. 
Suppose that $\mathfrak{g}$ is semisimple and let $$\mathfrak{g}=\mathfrak{g}_{1,k_1}\oplus \cdots\oplus \mathfrak{g}_{r, k_r},$$ where  $\mathfrak{g}_{i,k_i}$'s are simple ideals of $\mathfrak{g}$ at level $k_i$.  Then the subVOA $U$ generated by $V_1$ is isomorphic to the tensor of simple affine VOAs  
\[
L_{\widehat{\mathfrak{g}_{1}}}(k_1,0) \otimes \cdots   \otimes L_{\widehat{\mathfrak{g}_{r}}}(k_r,0) 
\]
and $U$ is a full subVOA of $V$, i.e, $U$ and $V$ have the same conformal element \cite{DM}.  It was shown in \cite{DW} that for each $1\leq i\leq r$,  the affine VOA 
$L_{\mathfrak{g}_{i}}(k_i,0)$ contains a lattice VOA $V_{\sqrt{k_i}Q^i_L}$, where $Q^i_L$ is the lattice spanned by the long roots of $\mathfrak{g}_i$. 

Set $Q_\g= \sqrt{k_1}Q^1_L\oplus \cdots \oplus \sqrt{k_r}Q^r_L$, $W=\mathrm{Com}_{V}( V_{ Q_\g})$ and $X =\mathrm{Com}_{V}(W)$.  
Then it is clear that 
$X\supset V_{ Q_\g}$ and 
$\mathrm{Com}_{V}(X)=W$ and $\mathrm{Com}_{V}(W)=X$. 
In this case, the VOA $X$ is an extension of the lattice VOA $V_{Q_\g}$ and hence, 
 $ X\cong V_{L_\mathfrak{g}}$ for some even lattice $L_{\mathfrak{g}} > Q_\g$.  Notice that $V_{L_\mathfrak{g}}$ has group-like fusion, i.e., all irreducible $V_{L_\mathfrak{g}}$-modules are simple current modules (cf. \cite[Corollary 12.10]{DL94}). In this case, the set of all irreducible modules $R(V_{L_\g})$ forms an abelian group with respect to the fusion product.    
Indeed,   
$R( V_{L_\mathfrak{g}})$ has the quadratic form $q$ : $R(V_{L_\mathfrak{g}})\to\Q/ \Z$  defined by 
\[
q(V_{\alpha+{L_\mathfrak{g}}})=\wt(V_{\alpha+{L_\mathfrak{g}}})= \frac{(\alpha|\alpha)}{2} \mod \Z, 
\]
where $\wt(\cdot)$ denotes the conformal weight of the module. Moreover, $R( V_{L_\mathfrak{g}})$ is isomorphic to $\mathcal{D}(L_\mathfrak{g})=L_\mathfrak{g}^*/L_\mathfrak{g}$ as a quadratic space.

By some recent results on coset constructions \cite[Main Theorem 2]{CKM} (see also  \cite{KMi} and \cite{Lin}), it is known that the VOA $W$ also has group-like fusion and 
$R(W)$ forms a quadratic space isomorphic to $(R( V_{L_\mathfrak{g}}), -q)$, where the quadratic form is defined by conformal weights modulo $\Z$. Since $V_{L_\g}\otimes W$ is a full subVOA of $V$ and $V$ is holomorphic, the VOA $V$ defines a maximal totally singular subspace of $R( V_{L_\mathfrak{g}})\times R(W)$; hence it induces an anti-isomorphism of quadratic spaces 
$\varphi: (R( V_{L_\mathfrak{g}}), q)  \to (R(W), q')$ such that $q(M)+q'(\varphi(M))=0$ for all $M\in R( V_{L_\mathfrak{g}})$. 

Conversely, let $\varphi: (R( V_{L_\mathfrak{g}}), q)  \to (R(W), -q')$ be an isomorphism of quadratic spaces. Then the set
$\{ (M, \varphi(M))\mid M\in R( V_{L_\mathfrak{g}})\}$ is a maximal totally singular subspace of $R( V_{L_\mathfrak{g}})\times R(W)$ and hence $U=\bigoplus_{M\in R( V_{L_\mathfrak{g}})} M\otimes \varphi(M)$ 
has a structure of a holomorphic VOA. 

H\"ohn noticed that the VOA $W$ seems to be related to a certain coinvariant sublattice of the Leech lattice $\Lambda$.  For $g\in O(\Lambda)$, the coinvariant lattice $\Lambda_g$ is defined to be the sublattice of $\Lambda$ which is orthogonal to all fixed points of $g$ in $\Lambda$ (see Definition \ref{def:2.2}).
In particular, H\"ohn proposed the following conjecture.
 
\begin{conjecture}\label{c:1.1}
For each semisimple case in Schellekens' list, there exists an isometry $g\in O(\Lambda)$ such that $(R(V_{\Lambda_g}^{\hat{g}}),q)\cong (R( V_{L_\mathfrak{g}}), -q)$ as quadratic spaces.
\end{conjecture}

The isometry $g$ for each case  has also been described by  H\"ohn (see  \cite[Table 4]{Ho2}). In this article, we will prove his conjecture (see Theorem \ref{Th2}). In fact, we will study a more general situation and prove  the following theorem (see Lemma \ref{onlyif} and Theorem \ref{tDaction}). Note that the case when $|g|$ is a prime has also been studied in \cite{LS17}. 

\begin{main} Let $L$ be an even lattice. Let $g\in O(L)$ and $\hat{g}$ a lift of $g$ in $\aut(V_L)$ with finite order. 
The VOA $V_{L}^{\hat{g}}$ has group-like fusion, i.e, all irreducible $V_{L}^{\hat{g}}$-modules are simple current modules 
if and only if  $g$ acts trivially on $L^*/L$  (or equivalently, $(1-g)L^*<L$). 
\end{main}

When $g$ is fixed point free on $L$ and acts trivially on $L^*/L$, we will also determine the quadratic space structure for $(R(V_{L}^{\hat{g}}), q)$, where $q$ is defined by the conformal weights modulo $\Z$ (see Section \ref{S:5}). 
In this case, $L$ can be embedded primitively into an even unimodular lattice $N$  and $g$ can be lift to an isometry of $N$ such that $L=N_{g}$ is the coinvariant sublattice of $g$ in $N$. Therefore, $(R(V_{L}^{\hat{g}}), q)$ can be determined by the corresponding structure of $(R(V_N^{\phi_g}), q)$ \cite{EMS} and the decomposition of $V_N^{\phi_g}$ as a sum of irreducible $
V_{N_g}^{\hat{g}}\otimes V_{N^g}$-modules, where $\phi_g$ denotes a standard lift of $g$ in $\Aut(V_N)$ (see Definition \ref{def:3.2}). There are several different cases which depend on the order of $\phi_g$ and the conformal weights of the corresponding twisted modules.   When $|\phi_g|=|g|$, we have $R(V_{L}^{\hat{g}})\cong L^*/L \times R(V_N^{\phi_g})$
as an abelian group (cf. Theorem \ref{thm:5.2}). The quadratic space structure of $R(V_{L}^{\hat{g}})$ can then be determined by using the decomposition of $V_N^{\phi_g}$ as a sum of irreducible $
V_{L}^{\hat{g}}\otimes V_{N^g}$-modules. When $|\phi_g|=2|g|=2p$, $\phi_g^p= \sigma_h$ for some $h\in L^*$ and $L^*= X\cup (u+X)$, where $X=\{x\in L^* \mid (x | h)\in \Z\}$. In this case, $R(V_{L}^{\hat{g}})\cong Y/L \times I(R(V_N^{\phi_g}))$
as an abelian group (cf. Theorems \ref{Rpg=2n}), where $Y=\{ y\in L^*\mid 
(h|y)\in \Z \text { and } (u|y)\in \Z\}$ and $I$ is a group homomorphism from $R(V_N^{\phi_g})$ to $R(V_{L}^{\hat{g}})$ (see Lemma \ref{mapI}). The corresponding quadratic space structure can also be determined by using the decomposition of $V_N^{\phi_g}$ as a sum of irreducible $
V_{L}^{\hat{g}}\otimes V_{N^g}$-modules; however, the decomposition also depends on the conformal weights of the twisted modules in this case  (see Cases 2a and 2b of Section \ref{S:5}).  By applying the result to the Leech lattice $\Lambda$, we determine the quadratic space structures for $R(V_{\Lambda_g}^{\hat{g}})$ for several conjugacy classes in $O(\Lambda)$ and verify the conjecture of H\"ohn \cite{Ho2} for these cases (see Section \ref{sec:6}).  We also describe a construction of a holomorphic VOA of central charge $24$ such that its weight one Lie algebra has the type $F_{4,6}A_{2,2}$ using H\"ohn's idea (see Section~\ref{sec:7}).   We remark that some of the strategies that we used here have also been discussed in \cite[Section 4.3]{CKM1} in a more general setting. 

The organization of this article is as follows. In Section 2, we review some basic notions for integral lattices. In Section 3, we review the construction of lattice VOAs and the structures of their automorphism group. We also recall a construction of irreducible twisted modules for (standard) lifts of isometries.
In Section 4, we compute the quantum dimensions for some irreducible twisted modules and prove Main Theorem 1. In Section 5, we determine the fusion  ring for $V_{L}^{\hat{g}}$ and compute the corresponding quadratic form. The main idea is to decompose every irreducible module of $V_N^{\phi_g}$ into a sum of irreducible $V_{L}^{\hat{g}}\otimes V_{N^g}$-modules. In Section 6,  we study certain explicit examples associated with the Leech lattice and verify the conjecture of H\"ohn. In Section 7, we discuss a construction of a holomorphic VOA of central charge $24$ whose weight one Lie algebra has the type $F_{4,6}A_{2,2}$ using the fusion rules of the VOA $V_{\Lambda_g}^{\hat{g}}$, where $g$ is an isometry of the conjugacy class $6G$ of $O(\Lambda)$.

\begin{remark}
Recall that the category of $V_{L_\g}$-modules is a pointed modular tensor category and a pointed modular tensor category is completely determined by its quadratic space structure. Therefore,  
H\"ohn's conjecture may also be viewed as an equivalence of modular tensor category, namely, there is a braided reversing  equivalence between the category of $V_{L_\g}$-modules and the category of $V_{\Lambda_g}^{\hat{g}}$-modules.   
\end{remark}
\section{Preliminary}

\subsection{Integral lattices}

By \textit{lattice}, we mean a  free abelian group of finite rank with  a rational
valued, positive definite symmetric bilinear form $(\cdot|\cdot)$. 
A lattice $L$  is \textit{integral} if $(L|L) < \Z$ and it is \textit{even} if 
$(x| x)\in 2\Z$ for any $x\in L$. 
We use $L^*$ to denote the dual lattice of $L$, $$L^*= \{v\in \Q\otimes_\Z L \mid (v| L) < \Z\},$$ 
and denote the discriminant group
$L^*/L$ by $\mathcal{D}(L)$. 
Note that if a lattice $L$ is integral, then $L\subset L^*$.
 Let $\{x_1, \dots, x_n\}$ be a basis of $L$. The Gram matrix of $L$ is defined to be the matrix $
G=\left(\, (x_i|  x_j)\,  \right)_{1\leq i, j\leq n}$. 
The determinant of $L$, denoted by $\det(L)$, is the determinant of $G$. Note that 
$|\det(L)|=|\mathcal{D}(L)|$.

Let $L$ be an integral lattice. For any positive integer $m$, let
$
L_m=\{ x\in L | \, ( x|x )=m\}
$
be the set of all norm $m$ elements in $L$.  
\textit{The summand of
$L$ determined by the subset $S$} of $L$  is the intersection of $L$
with the $\Q$-span of $S$. 
An \textit{isometry} $g$ of $L$ is a linear isomorphism $g\in GL(\Q\otimes_\Z L)$ such that $g(L)\subset L$ and $(gx | gy)=(x|y)$ for all $x,y\in L$. We denote the group of all isometries of $L$ by $O(L)$.

\begin{definition}
Let $p$ be a prime. An integral lattice $L$ is said to be \textit{$p$-elementary} if $ p L^* < L$. A $1$-elementary lattice is also called \textit{unimodular}. 
\end{definition}

\begin{definition}\label{def:2.2}
Let $L$ be an integral lattice and $g\in O(L)$. We denote the fixed point sublattice of $g$ by
$L^g=\{x\in L\mid gx=x\}.$  

The \textit{coinvariant lattice} of $g$ is defined to be 
\[
L_g = Ann_L(L^g) = \{ x\in L \mid ( x| y) =0 \text{ for all } y\in L^g\}.
\]
\end{definition}

First we recall the following simple observation (cf. \cite{Nik} and \cite{LS17}). 

\begin{lemma}\label{exp}
Let $L$ be an even unimodular lattice. Let $g\in O(L)$ be an isometry of order $\ell>1$ such that $L^g\neq 0$. Then $\ell(L^g)^* < L^g$ and $\mathcal{D}(L^g)\cong \mathcal{D}(L_g)$. 
\end{lemma}

By the above lemma, we have $\ell \lambda\in L_g$ for any $\lambda\in L_g^*$ and hence the exponent of the group $L_g^*/L_g$ divides $\ell$.

\section{Lattice VOAs, automorphisms and twisted modules}
In this section, we review the construction of a lattice VOA and the structure of its automorphism group from \cite{FLM,DN}. We also review a construction of irreducible twisted modules for (standard) lifts of isometries from \cite{Le,DL} (see also \cite{BK04}). 

\subsection{Lattice VOA and the automorphism group}

Let $L$ be an even lattice of rank $m$ and let $(\cdot |\cdot )$ be the positive-definite symmetric bilinear form on $\R\otimes_\Z L\cong\R^m$.
The lattice VOA $V_L$ associated with
$L$ is defined to be $M(1) \otimes \C\{L\}$. 
Here $M(1)$ is the Heisenberg VOA associated with $\mathfrak{h}=\C\otimes_\Z L$ and the form $(\cdot|\cdot)$ extended $\C$-bilinearly. That $\C\{L\}=\bigoplus_{\alpha\in L}\C e^\alpha$ is the twisted group algebra with commutator relation $e^\alpha e^\beta=(-1)^{(\alpha|\beta)}e^{\beta}e^{\alpha}$, for  $\alpha,\beta\in L$.
We fix a $2$-cocycle $\varepsilon(\cdot|\cdot):L\times L\to\{\pm1\}$ for $\C\{L\}$ such that $e^\alpha e^\beta=\varepsilon(\alpha|\beta)e^{\alpha+\beta}$, $\varepsilon(\alpha|\alpha)=(-1)^{(\alpha|\alpha)/2}$ and $\varepsilon(\alpha|0)=\varepsilon(0|\alpha)=1$ for all $\alpha,\beta\in L$.

Let $\hat{L}$ be the central extension of $L$ by $\langle-1\rangle$ associated with the $2$-cocycle $\varepsilon(\cdot|\cdot)$.
Let $\Aut \hat{L}$ be the set of all group automorphisms of $\hat L$.
For $\varphi\in \Aut \hat{L}$, we define the element $\bar{\varphi}\in\Aut L$ by $\p(e^\alpha)\in\{\pm e^{\bar\p(\alpha)}\}$, $\alpha\in L$.
Set $$O(\hat{L})=\{\p\in\Aut\hat L\mid \bar\p\in O(L)\}.$$
For $\chi\in\mathrm{Hom}(L,\Z_2)$, the map $\hat{L}\to\hat{L}$, $e^{\alpha}\mapsto (-1)^{\chi(\alpha)}e^{\alpha}$, is an element in $O(\hat{L})$.
Such automorphisms form an elementary abelian $2$-subgroup of $O(\hat{L})$ of rank $m$, which is also denoted by $\mathrm{Hom}(L,\Z_2)$.
It was proved in \cite[Proposition 5.4.1]{FLM} that the following sequence is exact:
\begin{equation}
1 \longrightarrow \mathrm{Hom}(L, \Z_2) { \longrightarrow}
O(\hat{L}) \bar\longrightarrow O(L)\longrightarrow  1.\label{Exact1}
\end{equation}
We identify $O(\hat{L})$ as a subgroup of $\Aut V_L$ as follows: for $\varphi\in O(\hat{L})$, the map  
$$\alpha_1(-n_1)\dots\alpha_m(-n_s)e^\beta\mapsto \bar{\varphi}(\alpha_1)(-n_1)\dots\bar{\p}(\alpha_s)(-n_s)\p(e^\beta)$$
is an automorphism of $V_L$, 
where $n_1,\dots,n_s\in\Z_{>0}$ and $\alpha_1,\dots,\alpha_s,\beta\in L$.

Note that we often identify $\mathfrak{h}$ with $\mathfrak{h}(-1)\1$ via $h\mapsto h(-1)\1$.

\begin{proposition}[{\cite[Theorem 2.1]{DN}}] \label{Prop:AutVLambda} The automorphism group $\Aut V_L$ of $V_L$ is generated by the normal subgroup $N(V_L)=\langle\exp({a_{(0)}})\mid a\in (V_L)_1\rangle$ and the subgroup $O(\hat L)$.
\end{proposition}

\subsection{Lifts of isometries of lattices}

\begin{definition}[\cite{Le} (see also \cite{EMS})] \label{def:3.2}
An element $\p\in \aut(V_L)$ is called a \emph{lift} of $g\in O(L)$ if $\p(e^\alpha)\in \C e^{g\alpha}$ for any $\alpha\in L$. 
A lift $\hat{g}$ of $g\in O(L)$ is said to be \emph{standard} if $\hat{g}(e^\alpha)=e^{\alpha}$ for all $\alpha\in L^g=\{\beta\in L\mid g(\beta)=\beta\}$.
\end{definition}

\begin{proposition}[{\cite[Section 5]{Le}}] For any isometry of $L$, there exists a standard lift.
\end{proposition}

The orders of standard lifts are determined in \cite[Lemma 12.1]{Bo92} (cf.\ \cite{EMS}) as follows:

\begin{lemma}[\cite{Bo92}]\label{Lem:ordSLift} Let $g\in O(L)$ be of order $n$ and let $\phi_g$ be a standard lift of $g$.
\begin{enumerate}[{\rm (1)}]
\item If $n$ is odd, then the order of $\phi_g$ is also $n$.
\item Assume that $n$ is even.
Then $\phi^n_g(e^\alpha)=(-1)^{(\alpha|g^{n/2}(\alpha))}e^\alpha$ for all $\alpha\in L$.
In particular, if $(\alpha|g^{n/2}(\alpha))\in2\Z$ for all $\alpha\in L$, then the order of $\phi_g$ is $n$; otherwise the order of $\phi_g$ is $2n$.
\end{enumerate}
\end{lemma}

\begin{remark}[See \cite{DL,EMS,LS6}]
A standard lift of an isometry is unique, up to conjugation in $\aut(V_L)$. 
\end{remark}

\subsection{Irreducible twisted modules for lattice VOAs}\label{Sec:twist}

Let $g\in O(L)$ be of order $p$ and $\phi_g\in O(\hat{L})$ a standard lift of $g$.
Set $n=|\phi_g|$, which is either $p$ or $2p$.

Let $(L^*/L)^g$ be the set of cosets of $L$ in $L^*$ fixed by $g$, i.e., $g\lambda+L =\lambda+L$ for $\lambda+L\in (L^*/L)^g$. Let $P_0^g:L^* \to \Q\otimes_\Z L^g$ be the orthogonal projection. 
Then  $V_L$ has exactly $|(L^*/L)^g|$ irreducible $\phi_g$-twisted $V_L$-modules, up to isomorphism (see \cite{DLM2}). The irreducible $\phi_g$-twisted $V_L$-modules have been constructed in \cite{Le,DL} explicitly and are classified in \cite{BK04}.
They are given by 
\begin{equation}
V_{\lambda+L}[\phi_g]= M(1)[g]\otimes\C[P_0^g(\lambda+L)]\otimes T_{\tilde{\lambda}}, 
\qquad \text{ for } \lambda+L\in (L^*/L)^g,   \label{twmodule0}
\end{equation}
where $M(1)[g]$ is the ``$g$-twisted" free bosonic space, $\C[\lambda+P_0^g(L)]$ is a module for the group algebra of $P_0^g(L)$ and $T_{\tilde{\lambda}}$ is an irreducible module for a certain ``$g$-twisted" central extension of $L_g$ associated with $\tilde{\lambda}=(1-g)\lambda$ (see \cite[Propositions 6.1 and 6.2]{Le} and \cite[Remark 4.2]{DL} for detail).
Note also that $M(1)[g]$ is spanned by vectors of the form
$$ x_1(-m_1)\dots x_s(-m_s)1,$$
where $m_i\in(1/p)\Z_{>0}$, $x_i\in\mathfrak{h}_{(pm_i)}$ for $1\le i\le s$, and 
$\mathfrak{h}_{(j)}=\h_{(j; g)}=\{x\in\h\mid g(x)=\exp((j/p)2\pi\sqrt{-1})x\}$.

Then  the conformal weight of $x_1(-m_1)\dots x_s(-m_s)\otimes e^\alpha\otimes t\in V_{\lambda+L}[\phi_g]$ is given by 
\begin{equation}
\sum_{i=1}^s m_i+\frac{(\alpha|\alpha)}{2}+\rho_T, \label{Eq:rho}
\end{equation}
where 
\[
\rho_T =  \frac{1}{4p^2}\sum_{j=1}^{p-1}j(p-j)\dim \mathfrak{h}_{(j)},
\]
$x_1(-m_1)\dots x_s(-m_s)\in M(1)[g]$, $e^\alpha\in\C[P_0^g(\lambda+L)]$ and $t\in T_{\tilde{\lambda}}$.
Notice that the minimal conformal weight of $V_{\lambda+L}[\phi_g]$ is given by 
$$
\frac{1}2\min\{(\beta|\beta)\mid \beta\in P_0^g(\lambda+L)\}+\rho_T.
$$

\subsection{Non-standard lifts}  Let $\hat{g}$ be an arbitrary lift of $g$. Then 
$\hat{g}\phi_g^{-1}$ acts trivially on $M(1)$ and $ \hat{g}\phi_g^{-1}(e^\alpha) \in \C e^\alpha$ for any $\alpha \in L$. In this case, $ \hat{g} \phi_g^{-1} =\sigma_h=\exp(-2\pi i h_{(0)})$ for some $h\in \mathfrak{h}$ (cf. \cite[Lemma 2.5]{DN}).  Let $\mu$ be the image of $h$ to $\mathfrak{h}_{(0)}$ under the orthogonal projection. Then by \cite[Lemma 4.5]{LS6},  $\hat{g}$ is conjugate to $\sigma_\mu \phi_g$ in $\aut(V_L)$,  where $\sigma_\mu= \exp(-2\pi i \mu_{(0)})$.
Without loss,  we may assume  $\hat{g}=\sigma_\mu \phi_g$. 

For simplicity, we always assume that $\hat{g}$ has finite order in this article. In this case, 
$\mu \in (\frac{1}k L^*)\cap \mathfrak{h}_{(0)}$ for some positive integer $k$.  

\medskip

We first recall the following result from \cite{Li}.

\begin{proposition}[{\cite[Proposition 5.4]{Li}}]\label{Prop:twist}
Let $g$ be an automorphism of $V$ of finite order and
let $h\in V_1$ such that $g(h) = h$. We also assume $h_{(0)}$ acts semisimply on $V$ and $\mathrm{Spec\,} h_{(0)} < ({1}/{k})\Z$ for a positive integer $k$. 
Let $(M, Y_M)$ be a $g$-twisted $V$-module and
define $(M^{(h)}, Y_{M^{(h)}}(\cdot, z)) $ as follows:
\[
\begin{split}
& M^{(h)} =M \quad \text{ as a vector space;}\\
& Y_{M^{(h)}} (a, z) = Y_M(\Delta(h, z)a, z)\quad \text{ for any } a\in V,
\end{split}
\]
where $\Delta(h, z) = z^{h_{(0)}} \exp\left( \sum_{n=1}^\infty \frac{h_{(n)}}{-n} (-z)^{-n}\right)$. 
Then $(M^{(h)}, Y_{M^{(h)}}(\cdot, z))$ is a
$\sigma_hg$-twisted $V$-module.
Furthermore, if $M$ is irreducible, then so is $M^{(h)}$.
\end{proposition}

By the proposition above,  the module $V_{\lambda+L}[\phi_g]^{(\mu)}$ is an irreducible $\hat{g}= \sigma_\mu \phi_g$-twisted $V_L$-module. The conformal weight of $x_1(-m_1)\dots x_s(-m_s)\otimes e^\alpha\otimes t$ in $V_{\lambda+L}[\phi_g]^{(\mu )}$ ($m_i\in(1/p)\Z_{>0}$, $\alpha\in P_0^g(\lambda+L)$ and $t\in T_{\tilde{\lambda}}$) is given by 
\begin{equation}
\sum_{i=1}^s m_i+\frac{(\mu +\alpha|\mu+\alpha)}{2}+\rho_T.\label{Eq:wttw}
\end{equation}
As a vector space,  we also have 
\begin{equation}
V_{\lambda+L}[\phi_g]^{(\mu)}\cong M(1)[g]\otimes\C[\mu + \lambda'+ P_0^{g}(L)]\otimes T_{\tilde{\lambda}}  \label{twmodule}
\end{equation}
For simplicity, we denote it by $V_{\lambda+L}[\hat{g}]$.

\section{Orbifold VOA  $V_{L}^{\hat{g}}$ having group-like fusion}
\begin{definition}[\cite{EMS}]
Let $V$ be a $C_2$-cofinite, rational VOA of CFT type. We say that $V$ has group-like fusion if 
all irreducible $V$-modules are simple current modules. In this case, the set of all inequivalent irreducible modules $R(V)$ forms a finite abelian group with respect to the fusion product and there is a quadratic form on $R(V)$ defined by conformal weights modulo $\Z$.
\end{definition}

Let $L$ be an even lattice and $g\in O(L)$. Let $\hat{g}$ be a lift of $g$ in $\aut(V_L)$ with finite order. 
In this section, we will study the orbifold VOA  $V_{L}^{\hat{g}}$. As our main result, we will show that $V_{L}^{\hat{g}}$ has group-like fusion if and only if $g$ acts trivially on $L^*/L$ (or equivalently, $(1-g)L^*<L$).

\subsection{Group-like fusion}

Let $V$ be a VOA and $f\in \mathrm{Aut}(V)$. For any irreducible module $M$ of $V$, we denote the $f$-conjugate of $M$ by $M\circ f$, i.e., $M\circ f=M$ as a vector space and the vertex operator $Y_{M\circ f}( u,z)= Y_M(fu,z)$ for $u\in V$.

If $V=V_L$ is a lattice VOA and $\hat{g}$ is a lift of an isometry $g\in O(L)$, then $V_{\alpha+L}\circ \hat{g} \cong V_{g^{-1}\alpha +L}$ for $\alpha+L\in L^*/L$.

\begin{proposition}[{\cite[Theorem 6.1]{DMq}}]\label{gconj}
Let $V$ be a simple VOA and $g\in \mathrm{Aut}(V)$. Let $M$ be an irreducible module of $V$ and let $1\leq i \leq |g|$ be the smallest integer such that $M \cong M\circ g^i$ as $V$-modules. Then 
\begin{enumerate}
\item $M$ decomposes into a sum of $|g|/i$ irreducible $V^g$-modules. 

\item $M\cong M\circ g^j$ as $V^g$-modules for any $j<i$. 
\end{enumerate}
 
\end{proposition}

\begin{lemma}\label{onlyif}
Suppose $V_{L}^{\hat{g}}$ has group-like fusion. Then $\alpha+L=g\alpha +L$ for all $\alpha+L\in L^*/L$.
\end{lemma}

\begin{proof}
Suppose that $\alpha+L\neq g\alpha+L$ for some $\alpha$. 
Then the integer $i$ defined in Proposition \ref{gconj} is strictly bigger than $1$ and $V_{\alpha+L} \cong V_{g^{-1}\alpha+L}$ as $V_{L}^{\hat{g}}$-modules. Let $W$ be an irreducible  $V_L^{\hat{g}}$-submodule of $V_{\alpha+L}$. Then there are non-zero $V_L^{\hat{g}}$-intertwining operators from $W\times W$ to $V_{2\alpha+L}$ and 
from $W\times W$ to $V_{\alpha + g^{-1}\alpha+L}$. It means that $W$ is not a simple current module and it contradicts that  $V_{L}^{\hat{g}}$ has group-like fusion. 
\end{proof}

By the Lemma above, it is necessary that $(1-g)L^*< L$ (i.e., $g$ acts trivially on $L^*/L$) if $V_{L}^{\hat{g}}$ has group-like fusion. It turns out that the condition $(1-g)L^*< L$ is also sufficient for proving that $V_{L}^{\hat{g}}$ has group-like fusion (see Theorem \ref{tDaction}).

Recall that $V_{L}^{\hat{g}}$ is  $C_2$-cofinite and rational \cite{Mi,CM}. It is also proved in \cite{DRX} that any irreducible $V_{L}^{\hat{g}}$-module is a submodule of an irreducible $\hat{g}^i$-twisted $V_{L}$-module for some 
$0\le i\le |\hat{g}|-1$. Therefore, it suffices to compute the quantum dimensions for irreducible $V_L^{\hat{g}}$-submodules of irreducible $\hat{g}^i$-twisted $V_{L}$-modules and to show that they all have the quantum dimension $1$ \cite{DJX} (see also Theorem \ref{qdim1}).

\subsection{Quantum dimensions of twisted modules of $V_L$}\label{Sect5.3}
In this subsection, we will  compute the quantum dimensions for irreducible  
$\hat{g}$-twisted modules for $V_L$.

We first recall some facts about quantum dimensions of irreducible modules of vertex operator algebras from \cite{DJX} and \cite{DRX}. Let $V$ be a strongly regular VOA.  Let $g\in \aut(V)$ and $M$ an irreducible  $g$-twisted $V$-module.
The {\em quantum dimension} of $M$ is defined to be
\begin{align*}
\qdim_V M=\lim_{y\to 0}\frac{Z_{M}(iy)}{Z_{V}(iy)},
\end{align*}
where $Z_M(\tau)=Z_M(\1,\tau)$ is the character of $M$ and $y$ is real and positive.  Note that an irreducible module of $V$ is an irreducible  $1$-twisted $V$-module. 

\medskip

The following result was proved in \cite{DJX}.

\begin{theorem}\label{qdim1}
Let $V$ be a strongly regular vertex operator algebra, $M^0=V, M^1,...,M^p$ be all the irreducible $V$-modules. Assume further that the conformal weights of $ M^1,...,M^p$ are greater than $0$. Then
\begin{enumerate}[{\rm (1)}]
\item $\qdim_V M^i\geq 1$ for any $0\leq i\leq p$;
\item $M^i$ is a simple current module if and only if $\qdim M^i=1$.
\end{enumerate}
\end{theorem}

The following lemma will be used in the computation of quantum dimensions. Note that the case when $|g|$ is a prime has been discussed in \cite{LS17}.

\begin{lemma}\label{detLg}
Let $L$ be a lattice (not necessarily integral).
Let $g\in O(L)$ be 
fixed point free 
isometry of $L$. Then we have $|L/(1-g)L|= |\det(1-g)|$.     
\end{lemma}

\begin{proof}
Let $N=(1-g)L$. Since $g$ is fixed point free on $L$, $(1-g)$ induces a $\Z$-linear isomorphism between $L$ and $N$ and $N$ is a full rank $\Z$-submodule of $L$. By the elementary factor theorem, there is a basis $\{x_1, \dots , x_n\}$ of $L$  and $\lambda_1, \dots, \lambda_n\in \Z_{>0}$ such that 
$\{ \lambda_1 x_1, \dots, \lambda_nx_n\}$ is a basis for $N$. Note that $\{ (1-g)x_1, (1-g)x_2, \dots, (1-g)x_n\}$ also forms a basis for $N$. Let $\phi$ be a $\Z$-linear map which maps 
$(1-g)x_i$ to $\lambda_ix_i$. Then the matrix of $\phi (1-g)$ with respect to the basis  $\{x_1, \dots , x_n\}$ is given by 
\[
\begin{pmatrix}
\lambda_1 & \cdots & 0\\
\vdots  & \ddots & \vdots\\
0 & \cdots & \lambda_n
\end{pmatrix} . 
\]
Then we have $[L:N]= \lambda_1\cdots\lambda_n=\det\phi(1-g)= \det\phi\det(1-g)$.  Since $\phi$ maps a $\Z$-basis of $N$ to a $\Z$-basis, we have $|\det\phi|=1$ and the lemma follows. 
\end{proof}


Next we review several facts about the character for $\phi_{g}$-twisted $V_L$-modules, where $\phi_g$ is a standard lift of $g$. First we assume that $g$ is a fixed point free isometry on $L$.  
In this case, $L^g=0$ and the irreducible $\phi_g$-twisted modules are given by 
$V_L^T=M(1)[g]\otimes T$, where $T$ is an irreducible module for a certain ``$g$-twisted" central extension of $L$ (see \cite{DL}).
As a consequence, we have  the following result.

\begin{lemma}[cf. \cite{ALY,DL}]
Let $g\in O(L)$ be fixed point free. Then 
\begin{align*}
Z_{V_L^T}(\tau)&=\frac{(\dim T)q^{\left(\sum_{i=1}^{p-1}\frac{i(p-i)r_i}{4p^2}\right)-\frac{\ell}{24}}}{\prod_{j=1}^{p-1}\prod_{n=0}^\infty (1-q^{j/p+n})^{r_j}}
=(\dim T)q^{\left(\sum_{i=1}^{p-1}\frac{i(p-i)r_i}{4p^2}\right)-\frac{\ell}{24}}\prod_{d|p}\frac{q^{\frac{m_d}{24d}}}{\eta(\tau /d)^{m_d}},
\end{align*}
where $\ell=\mathrm{rank}(L)$, $r_i=\dim \mathfrak{h}_{(j)}$ and $q=e^{2\pi i\tau}$. 
\end{lemma}

Since the character of $V_{L}$ is given by 
\begin{align*}
Z_{V_L}(\tau)=\frac{\Theta_L(\tau)}{\eta(\tau)^\ell},
\end{align*}
where $\Theta_L(\tau)$ is the theta function of $L$, we have  
\begin{equation}
\frac{Z_{V_L^T}(iy)}{Z_{V_L}(iy)}=(\dim T) q^{\left(\sum_{i=1}^{p-1}\frac{i(p-i)r_i}{4p^2}\right)-\frac{\ell}{24}+\sum_{d|p}\frac{m_d}{24d}}\times \frac{\eta(iy)^\ell}{\prod_{d|p}\eta(iy/d)^{m_d}}\frac{1}{\Theta_L(iy)}.\label{E:VLT}
\end{equation}

The following result can be found in \cite{ALY}.

\begin{theorem}\label{theoremqdimtw_standard}
Let $L$ be an even lattice of rank $\ell$.
Let $g$ be a fixed point free isometry of $ L$. 
Let $\hat{g}$ be a lift of $g$.
For any $\widehat{L}_g$-module $T$, the quantum dimension of the $\hat{g}$-twisted module $V_L^T$  
exists and 
\[
\qdim_{V_L}V_L^T:=\lim_{y\rightarrow 0^+}\frac{Z_{V_L^T}(iy)}{Z_{V_L}(iy)}=\frac{v\dim T}{\prod_{d|p}d^{m_d/2}},
\]
where $v=\sqrt{|\mathcal{D}(L)|}$ and $m_d$ are integers given by 
$\det(x-g)=\prod_{d|p}(x^d-1)^{m_d}$. 
\end{theorem}

As a corollary, we have 
\begin{corollary}\label{cor973}
Let $L$ be an even lattice.  Assume that $g\in O(L)$ is fixed point free and $(1-g)L^* <L$. Then
\[
\qdim_{V_{L_g}} V_{L}^T=  1 
\]
for any irreducible $(\widehat{L})_g$-module $T$.    
\end{corollary}
\begin{proof}
The proof is similar to \cite[Corollary 4.11]{LS17}. 
First we recall that $\dim(T)=|L/R|^{\frac{1}2} $, where $R=((1-g)L^*) \cap L$ (see \cite[Lemma 3.2]{ALY}).  

By our assumption, $R=((1-g)L^*) \cap L =(1-g)L^*$. 
Then by Theorem \ref{theoremqdimtw_standard} and Lemma \ref{detLg},  we have 
\[
\begin{split}
(\qdim_{V_{L}} V_{L}^T)^2& =\frac{1}{\prod_{d|p}d^{m_d}} \left(\left| \frac{L^*}{L}\right| \cdot \dim(T)^2\right)\\ 
&= \frac{1}{\prod_{d|p}d^{m_d}}  \left(\left| \frac{(1-g)L^*}{(1-g)L}\right| \cdot 
\left| \frac{L}{(1-g)L^*}\right|
\right)\\
&= \frac{1}{\prod_{d|p}d^{m_d}}  \cdot 
\left| \frac{L}{(1-g)L}\right|\\
&= \frac{1}{\prod_{d|p}d^{m_d}} \cdot |\det (1-g)|\\
& =  \frac{1}{\prod_{d|p}d^{m_d}} \cdot  \prod_{d|p}d^{m_d}= 1   
\end{split}
\]
as desired.  
\end{proof}

For any $1\leq i \leq |g|-1$, we have $L^{g^i}\perp L_{g^i} < L$ as a full rank sublattice. 
Let $p_1: L\otimes_\Z\Q \to L^{g^i}\otimes_\Z\Q$ and $p_2: L\otimes_\Z\Q \to L_{g^i}\otimes_\Z\Q$ be natural projections. We also use $\lambda'$ and $\lambda''$ to denote the images of $\lambda$ under the natural projections $p_1$ and $p_2$, respectively. Then 
\[
V_L= \bigoplus_{\lambda \in L/(L^{g^i}\perp L_{g^i})} V_{\lambda'+L^{g^i}} \otimes V_{\lambda''+ L_{g^i}}.
\] 

\begin{lemma}\label{gtoitrivial}
Let $L$ be an even lattice and let $g\in O(L)$. Suppose that  $(1-g)L^* <L$. Then $(1-g^i)L_{g^i}^* < L_{g^i}$ for any $1\leq i \leq |g|-1$. 
\end{lemma}

\begin{proof}
We first observe that $p_2(L^*)= L_{g^i}^*$. Thus, $(1-g^i)L^*= (1-g^i)L_{g^i}^*$. By our assumption, 
$(1-g^i)L^*< (1-g)L^* <L$. Therefore, $(1-g^i)L_{g^i}^* < (L_{g^i}\otimes \Q) \cap L =L_{g^i}$ as desired.  
\end{proof}

\begin{lemma}\label{lem:4.10}
Let $L$ be an even lattice. Let $g\in O(L)$ and  $\hat{g}$  a lift of $g$ in $\aut(V_L)$ with finite order.  Assume that  $(1-g)L^* <L$.  Then 
\[
 \qdim_{V_{L}} V_{\lambda+ L}[\hat{g}^i]= 1
\]
for any $\lambda+L\in L^*/L$ and $i=0, 1, \dots, |\hat{g}|-1$. 
\end{lemma}

\begin{proof}
Suppose $\hat{g}^i =\sigma_\mu \phi_{g^i}$. Then by  \eqref{twmodule}, the irreducible twisted module 
\[
V_{\lambda+L}[\hat{g}^i]= M(1)[g^i]\otimes\C[\mu + \lambda'+ P_0^{g^i}(L)]\otimes T_{\tilde{\lambda}}. 
\]
Note that for $i=0$, $\hat{g}^i=1$, $\mu=0$ and $V_{\lambda+L}[\hat{g}^0]= M(1)\otimes\C[\lambda+L ]=V_{\lambda+L}.$

As a $\hat{g}^i$-twisted module of $V_{L^{g^i}}\otimes V_{L_{g^i}}$, we have 
\[
V_{\lambda+L}[\hat{g}^i]= \bigoplus_{\lambda \in L/(L^{g^i}\perp L_{g^i})} V_{\mu+\lambda'+L^{g^i}}\otimes V_{\lambda''+L_{g^i}} [\hat{g}^i]. 
\]
Then 
\[
\begin{split}
\frac{Z_{V_{\lambda+L}[\hat{g}^i]}(iy)}{Z_{V_{L}}(iy)} =
\frac{ \sum_{\lambda \in L/(L^{g^i}\perp L_{g^i})} Z_{V_{\mu+\lambda'+L^{g^i}}}(iy) Z_{V_{\lambda''+L_{g^i}} [\hat{g}^i]}(iy) }
{\sum_{\lambda \in L/(L^{g^i}\perp L_{g^i})} Z_{V_{\lambda'+L^{g^i}}}(iy) Z_{V_{\lambda''+L_{g^i}}}(iy)}.
\end{split}
\]
Divide both the numerator and the denominator by $Z_{V_{L^{g^i}}}(iy)Z_{V_{L_{g^i}}}(iy)$ and let $y$ tend to $0^+$. Then by Lemma \ref{gtoitrivial} and  Corollary \ref{cor973},  we have 
\[
\qdim_{V_{L}} V_{\lambda+ L}[\hat{g}^i]= \lim_{y\to 0^+} \frac{Z_{V_{\lambda+L}[\hat{g}^i]}(iy)}{Z_{V_{L}}(iy)}=\frac{[L: L^{g^i}\perp L_{g^i}]}{[L: L^{g^i}\perp L_{g^i}]}= 1
\]
as desired. 
\end{proof}

\begin{theorem}\label{tDaction} Let $L$ be an even lattice. Let $g\in O(L)$ and let $\hat{g}$ be a lift of $g$ in $\aut(V_L)$ with finite order. 
Then the VOA $V_{L}^{\hat{g}}$ has group-like fusion if $(1-g)L^*<L$. 
\end{theorem}

\begin{proof} First we recall that $V_{L}^{\hat{g}}$ is  $C_2$-cofinite and rational \cite{Mi,CM} and  any irreducible $V_{L}^{\hat{g}}$-module is a submodule of an irreducible $\hat{g}^i$-twisted $V_{L}$-module for some $0\le i\le |\hat{g}|-1$  \cite{DRX}.
Note also that the conformal weights of irreducible irreducible $\hat{g}^i$-twisted $V_{L}$-modules are positive for $1\le i\le |\hat{g}|-1$. Therefore, the conformal weights of irreducible $V_{L}^{\hat{g}}$-modules are also positive  except for $V_{L}^{\hat{g}}$ itself. 

Let $M$ be an irreducible $\hat{g}^i$-twisted 
$V_{L}$-module for some $0\le i\le |\hat{g}|-1$. Then by Lemma \ref{lem:4.10}, $\qdim_{V_{L}} M=1$.  It follows from \cite[Corollary 4.5]{DRX} that $$\qdim_{V_{L}^{\hat{g}}} M=|\hat{g}|.$$
For $i=0$, it follows from our assumption that all irreducible $V_{L}$-modules are $\hat{g}$-stable. For $1\leq i\leq |\hat{g}|-1$, it is also known \cite{DL,FLM} that all irreducible 
$\hat{g}^i$-twisted $V_{L}$-modules are $\hat{g}$-stable. Hence, the eigenspace decomposition of $\hat{g}$ on any $\hat{g}^i$-twisted  $V_{L}$-module $M$ gives a direct sum of $|\hat{g}|$ irreducible $V_{L}^{\hat{g}}$-submodules of $M$.
By Theorem \ref{qdim1} (1),  the quantum dimension of any irreducible $V_L^{\hat{g}}$-module is $\geq 1$. Thus, every irreducible $V_{L}^{\hat{g}}$-submodule of $M$ has quantum dimension $1$, and it is a simple current module. 
Hence all irreducible $V_{L}^{\hat{g}}$-modules are simple current modules.
\end{proof}

\begin{remark}\label{Num_mod}
By the discussion above, we know that there are exactly $|L^*/L|$ irreducible $\hat{g}^i$-twisted $V_L$-modules for each $0 \le i\le |\hat{g}|-1$ and each irreducible $\hat{g}^i$-twisted $V_L$-module decomposes as a direct sum of $|\hat{g}|$ irreducible $V_L^{\hat{g}}$-modules. Therefore, there are totally $|L^*/L|\cdot |\hat{g}|^2$ irreducible modules for $V_{L}^{\hat{g}}$. 
\end{remark}

It turns out that the main assumption $(1-g)L^*<L$ always holds if $L$ is a coinvariant sublattice of an even unimodular lattice. 

\begin{lemma}[cf. {\cite[Lemma 4.2]{LS17}}]\label{L4.2}
Let $L$ be an even unimodular lattice and $g\in O(L)$. Suppose $g\neq 1$. 
Then $(1-g)L_g^* < L_g$ and hence $\alpha+L_g = g\alpha +L_g$ for all $\alpha +L_g\in L_g^*/L_g$. 
\end{lemma}

As a corollary, we also have the following theorem. 

\begin{theorem}\label{glf_p} Let $L$ be an even unimodular lattice.
Let $g$ be an element in $O(L)$.  Then the VOA $V_{L_g}^{\hat{g}}$ has group-like fusion, namely, all irreducible modules of $V_{L_g}^{\hat{g}}$ are simple current modules.  
\end{theorem}

\begin{remark}\label{embed}
Let $L$ be a (positive definite) even lattice. Let $g\in O(L)$ be a fixed point free isometry such that $(1-g)L^*<L$. By 
\cite[Corollary 1.12.3]{Nik}, there is a (positive definite) even unimodular lattice $N$ such that $L$ can be embedded primitively into N (i.e., $(L\otimes_\Z \Q )\cap N =L$). 

Let $K=Ann_N(L)=\{x\in N\mid (x|L)=0\}$. Then $K\perp L$ is a full rank sublattice of $N$. Since $g$ acts trivially on $L^*/L$, the map $\tilde{g}$ given by $\tilde{g}|_K=id$ and $\tilde{g}|_{L}=g$ defines an isometry on $N$. Moreover, $L=N_{\tilde{g}}$ is the coinvariant sublattice of $g$ in $N$. 
\end{remark}

\section{Fusion ring of $V_{L}^{\hat{g}}$ and the corresponding quadratic space}\label{S:5}
In this section, we will determine the group structure for the fusion group of $V_{L}^{\hat{g}}$ and the corresponding quadratic space if $(1-g)L^*<L$ . We first  note that several special cases have been studied. 

(1) In \cite{ADL}, the fusion rules among irreducible $V_L^{\hat{g}}$-modules are determined when $|g|=2$ and $g$ is fixed point free on $L$. 

(2)  In \cite{EMS}, the case when $L$ is even unimodular has been studied. The fusion ring for $V_L^{\hat{g}}$ was also determined. 

For simplicity, we only consider the case when $g$ is fixed point free on $L$ (and $(1-g)L^*<L$). In this case, $L$ can be realized as a certain coinvariant sublattice of an even unimodular lattice $N$ (cf. Remark \ref{embed}) and we can use the results in \cite{EMS} to determine the fusion ring for $V_{L}^{\hat{g}}$.

 

\subsection{Fusion ring of $V_{N}^{\phi_{g}}$}\label{sec:5.1}

Next we  recall some facts about the fusion group of $V_N^{\phi_g}$ and the corresponding quadratic space structure from \cite{EMS} when $N$ is an even unimodular lattice. 

Let $N$ be an even unimodular lattice. Then the lattice VOA $V_N$ is holomorphic. 
Let $n$ be the order of $\phi_g$. Then for each $0 \leq i \leq n-1$, there is a unique irreducible $\phi_g^i$-twisted $V_N$-module  $V_{N}[\phi_{g}^i]$. The group $\langle \phi_g\rangle$ acts  naturally  on  $V_N[\phi_{g}^i]$ and such an action is unique up to a multiplication  of an $n$-th root of unity.   
Let $\varphi_i$ be a representation of $\langle \phi_g\rangle$ on $V_N[\phi_{g}^i]$. 
Denote 
\[
W^{i,j} =\{ w\in V_N[\phi_{g}^i]\mid \varphi_i(\phi_g) x= e^{2\pi \sqrt{-1} j/n} x\}
\] 
for $i,j\in \{0, \dots, n-1\}$.  In \cite{EMS}, it is proved that the orbifold VOA  $V_N^{\phi_g}$ has group-like fusion. Moreover, one can choose the representations $\varphi_i$, $i=0,\dots, n-1$, such that the fusion product 
\begin{equation}\label{FWij}
W^{i,k}\fusion W^{j, \ell} = W^{i+j, k+\ell +c_d(i,j)}, 
\end{equation}
where $c_d$ is defined by 
\[
c_d(i,j)= 
\begin{cases}
0 & \text{ if } i+j <n, \\
d & \text{ if } i+j \geq n
\end{cases}
\] 
for $i,j \in \{0, \dots, n-1\}$ and $d$ is determined by the conformal weight $\rho$ of the irreducible twisted module $V_N[\phi_g]$. More precisely, $d= 2n^2\rho \mod n$. 

In addition, it was proved that the conformal weight of $W^{i,j}$ defines a quadratic form 
\begin{equation}\label{qij}
q(i,j) \equiv \frac{ij}{n}+ \frac{i^2 t}{n^2} \mod \Z
\end{equation}
where $t\in \{0, 1, \dots, n-1\}$ and $t= n^2\rho \mod n$. In particular, $d=2t\mod n$. 

In this case,  the fusion algebra of $V_N^{\phi_g}$ is isomorphic to the group algebra $\C[D]$, where $D$ is an abelian group defined by a central extension 
\[
1\to \Z_n \to D \to \Z_n \to 1
\]
associated with the commutator map $c_d$. The abelian group $D$ is isomorphic to 
\begin{equation}\label{CD}
\Z_{n^2/(n,d)}\times \Z_{(n,d)}, 
\end{equation}
where $(n,d)$ denotes the gcd of $n$ and $d$.
Notice that $q$ also induces a quadratic form on $D$.

\subsection{Fusion ring of $V_{L}^{\hat{g}}$}

Let $L$ be  an even lattice and $g\in O(L)$. From now on, we always assume that $g$ is fixed point free on $L$ and $(1-g)L^*<L$. 

By \cite{Nik}, $L$ can be primitively embedded into an even unimodular lattice $N$. Moreover, $g$ can be lift to an isometry of $N$ and $g$ acts trivially on $K=Ann_N(L)$, i.e., $K=N^g$ and $L=N_g$ (cf. Remark \ref{embed}).  In this case, $\mathcal{D}(L) \cong \mathcal{D}(K)$ and  there is an isomorphism of the discriminant groups $f: \mathcal{D}(L) \to \mathcal{D}(K)$ such that 
\[
V_N= \bigoplus_{\lambda\in L^*/L}  V_{\lambda+L}\otimes V_{f(\lambda)+K}.
\]

Let $\phi_g$ be a standard lift of $g$ in $\aut(V_N)$. By definiton,  $\phi_g$ acts trivially on $V_K$ and it stabilizes the subVOA $V_{L}$. Then $\hat{g} = \phi_g|_{V_{L}}$ defines an automorphism in $\Aut(V_{L})$ and $V_N^{\phi_g}$ contains a full subVOA  $ V_{L}^{\hat{g}}\otimes V_{K} $. Moreover, $(V_{L}^{\hat{g}}, V_{K})$ forms a dual pair in $V_N^{\phi_g}$,  i.e., $\mathrm{Com}_{V_N^{\phi_g}} (V_{K})= V_{L}^{\hat{g}}$ and $\mathrm{Com}_{V_N^{\phi_g}} (V_{L}^{\hat{g}}) = V_{K}$. Therefore, any irreducible $V_{L}^{\hat{g}}$-module can be realized as a submodule of an irreducible $V_{N}^{\phi_g}$-module \cite{KMi} and one can compute the fusion rules  for  $V_{L}^{\hat{g}}$ by using the fusion rules of $V_{N}^{\phi_g}$.


By Lemma \ref{Lem:ordSLift},  $\phi_g$ has order $p=|g|$ if $p$ is odd. When $p$ is even,   
$\phi_g$ has order $p$ if $( x|  g^{p/2}(x))\in 2\Z$ for all $x\in N$ and  $\phi_g$ has order $2p$ if $( x| g^{p/2}(x)) \in 1+2\Z$ for some $x\in N$. 

As an easy observation, we notice that the order of $\phi_g$ depends only on the pair $(L,g)$. 
Recall that a lattice $L$ is said to be \emph{doubly even} if $(x|x)\in 4\Z$ for any $x\in L$.

\begin{lemma}\label{Lem:ordSLift1}
Let $L$, $g$, $N$ and $K$ as defined as above. Let $\phi_g$ be a standard lift of $g$ in $\aut(V_N)$. Suppose $p=|g|$ is even. 
Then 
\[
|\phi_g| =
\begin{cases}
|g| & \text{ if }\ 2L_{g^{p/2}}^* \text{ is doubly even},\\
2|g| & \text{ if }\  2L_{g^{p/2}}^* \text{ is not doubly even}.
\end{cases}
\]
\end{lemma}

\begin{proof}
Let $\varphi=g^{p/2}$. Then $|\varphi|=2$. By Lemma \ref{exp},  $2(N_\varphi)^* < N_\varphi$ and $N/(N^\varphi\perp N_\varphi)$ is an elementary abelian $2$-group. Note that $N_{\varphi} =L_{\varphi}$.  

For any $x\in N$, let $x'$ and $x''$ be the image of $x$ under the natural projections from $N$ to  $(N_\varphi)^*$  and
$(N^\varphi)^*$, respectively.  Then 
\[
(x|\varphi x) = -(x'|x')+ (x''|x'') \equiv 2(x'|x') \mod 2
\]    
because $(x|x)= (x'|x')+ (x''|x'')\equiv 0 \mod 2$.  

Since the projection from $N$ to  $(N_\varphi)^*$ is surjective,  $\phi_g$ has order $p$ if and only if $2(x'|x')\equiv 0 \mod 2$ for all $x'\in (N_\varphi)^*= (L_\varphi)^*$, i.e.,  $2L_\varphi^*$ is doubly even.  
\end{proof}


\medskip

Next we shall decompose irreducible $\phi_g^i$-twisted modules as  a sum of irreducible $ V_{L}^{\hat{g}}\otimes V_{K} $-modules. 

\subsubsection{ \textbf{Case 1: }$|\phi_g|=|g|=p$} In this case,  we have 
\[
\begin{split}
V_{N}[\phi_{g}^i] & =\bigoplus_{\lambda\in L^*/L} V_{\lambda+L}[\hat{g}^i]\otimes V_{f(\lambda)+K} 
\end{split}
\] 
for each $0\leq i\leq p-1$, where   $f: \mathcal{D}(L) \to \mathcal{D}(K)$ is an isomorphism of the discriminant groups such that $V_N= \bigoplus_{\lambda\in L^*/L}  V_{\lambda+L}\otimes V_{f(\lambda)+K}$.  
Note also that the conformal weight $\rho$ of the unique irreducible $\phi_g$-twisted $V_N$-module 
$V_{N}[\phi_{g}]$ is the same as the conformal weight of the irreducible $\hat{g}$-twisted $V_L$ module $V_{L}[\hat{g}]$.

For each $\phi_g$-invariant subspace $M$ and $0 \leq j < p-1$, we denote $$
M(j) =\{ w\in M \mid \phi_g(w) = e^{2\pi \sqrt{-1} j/p} w\}.$$  
Therefore, we have 
\[
V_N^{\phi_g} =  \bigoplus_{\lambda\in L^*/L} (V_{\lambda+L}\otimes V_{f(\lambda)+K})^{\phi_g} 
\]
and 
\begin{equation}\label{dwij}
\begin{split}
W^{i,j}= V_{N}[\phi_{g}^i](j) =\bigoplus_{\lambda\in L^*/L} (V_{\lambda+L}[\hat{g}^i]\otimes V_{f(\lambda)+K})(j). 
\end{split}
\end{equation}

It is clear that  the eigenspace 
$(V_{\lambda+L}[\hat{g}^i]\otimes V_{f(\lambda)+K})(j)$ is an irreducible $ V_{L}^{\hat{g}}\otimes V_{K}$-module for any $0 \leq j< p$ and $\lambda\in L^*$.  
Therefore, there exist $0\leq k < p$ such that 
\[
\begin{split}
&(V_{\lambda+L}[\hat{g}^i]\otimes V_{f(\lambda)+K})(j)\cong V_{\lambda+L}[\hat{g}^i](k) \otimes V_{f(\lambda)+K}. 
\end{split}
\]

By adjusting the action of $\hat{g}$ on $V_{\lambda+L}[\hat{g}^i]$ if necessary, we may assume  
\[
\begin{split}
(V_{\lambda+L}[\hat{g}^i]\otimes V_{f(\lambda)+K})(j)\cong V_{\lambda+L}[\hat{g}^i](j) \otimes V_{f(\lambda)+K}.
\end{split}
\]  

\medskip

\begin{lemma}\label{mapIp}
Define $I: R(V_N^{\phi_g}) \to R(V_{L}^{\hat{g}})$ such that $I(W^{i,j}) = V_{L}[\hat{g}^i]( j)$. Then $I$ is an injective  group homomorphism. 
\end{lemma}

\begin{proof}
It follows from formulas \eqref{FWij} and \eqref{dwij}. 
\end{proof}

\begin{theorem}\label{thm:5.2}
Suppose that $|g|=p$ is odd or $2L_{g^{p/2}}^*$ is doubly even if $p$ is even. Then  
\[
R(V_{L}^{\hat{g}})\cong L^*/L \times \Z_{p^2/(p,d)}\times \Z_{(p,d)}
\]
as an abelian group and  the quadratic form is given by  
\[
q(V_{\lambda+L}[\hat{g}^i](j)) \equiv \frac{ij}p +\frac{i^2t}{p^2} +\frac{(\lambda| \lambda)}2 \mod \Z, 
\]
where $t\in \{0, 1, \dots, p-1\}$ and $t= p^2\rho \mod p$,  $d= 2p^2\rho \mod p$, and $\rho$ is the conformal weight  of the irreducible $\hat{g}$-twisted module
$V_L[\hat{g}]$. In particular, we have 
\[
(R(V_{L}^{\hat{g}}),q) \cong (\mathcal{D}(L), q) \times (R(V_N^{\phi_g}), q) 
\]
as quadratic spaces. 
\end{theorem} 

\begin{proof}
Recall that $\phi_g$ has order $p=|g|$ under our hypothesis (cf.  Lemmas \ref{Lem:ordSLift} and \ref{Lem:ordSLift1}) and we have  
\[
\begin{split}
V_N^{\phi_g} &=  \bigoplus_{\lambda\in L^*/L} (V_{\lambda+L}\otimes V_{f(\lambda)+K})^{\phi_g} , \\
&=  \bigoplus_{\lambda\in L^*/L}  V_{\lambda+L}^{\hat{g}}\otimes V_{f(\lambda)+K} .  
\end{split}
\]
Then as a module for $ V_{L}^{\hat{g}}\otimes V_{K} $, 
\[
\begin{split}
W^{i,j} &= 
V_N^{\phi_g} \fusion_{ V_{L}^{\hat{g}} \otimes V_{K}} ( I(W^{i,j}) \otimes V_{K}), \\
&= \bigoplus_{\lambda\in L^*/L}  \left(V_{\lambda+L}^{\hat{g}}\fusion_{V_{L}^{\hat{g}}} I(W^{i,j})\right)\otimes V_{f(\lambda)+ K} .  
\end{split}
\]
By the decomposition of $W^{i,j}$ as irreducible $V_{L}^{\hat{g}}\otimes V_{K} $-modules (cf. \eqref{dwij}), we have 
\begin{equation}\label{dWijp}
V_{\lambda+L}^{\hat{g}}\fusion_{V_{L}^{\hat{g}}} I(W^{i,j}) = V_{\lambda+ L}[\hat{g}^i]( j).
\end{equation}
Now define $\varphi:  L^*/L \times R(V_N^{\phi_g}) \to R(V_{L}^{\hat{g}})$ by 
\[
\varphi( \lambda+L, W^{i,j}) =  V_{\lambda+L}^{\hat{g}} \fusion_{V_{L}^{\hat{g}}} I(W^{i,j}). 
\]
It follows from the fusion rules of $V_{L}$ and $V_N^{\phi_g}$ that $\varphi$ is a group homomorphism. Moreover, it is injective by Lemma \ref{mapIp} and $R(V_N^{\phi_g})\cong \Z_{p^2/(p,d)}\times \Z_{(p,d)}$ by \eqref{CD}.

For the quadratic form, we have $q(W^{i,j}) \equiv \frac{ij}{p}+ \frac{i^2 t}{p^2} \mod \Z$ (cf. \eqref{qij}) and 
\begin{equation}\label{qwijn=p}
W^{i,j}=\bigoplus_{\lambda\in L^*/L} V_{\lambda+L}[\hat{g}^i](j)\otimes V_{f(\lambda)+K}.
\end{equation}
Therefore, 
$$q(I(W^{i,j}))= q( V_{L}[\hat{g}^i](j))= q(W^{i,j}) \equiv \frac{ij}{p} + \frac{i^2t}{p^2} \mod \Z$$
and 
\begin{equation}\label{qform1}
q(V_{\lambda+L}[\hat{g}^i](j))\equiv \frac{ij}p +\frac{i^2t}{p^2} -\frac{(f(\lambda)| f(\lambda) )}2  \equiv \frac{ij}p +\frac{i^2t}{p^2} +\frac{(\lambda| \lambda)}2 \mod \Z. 
\end{equation}

Note that $(\lambda| \lambda)+(f(\lambda)|f(\lambda))\in 2\Z$  and $$\varphi(\lambda+L, W^{i,j}) = V_{\lambda+L}^{\hat{g}} \fusion_{V_{L}^{\hat{g}}} I(W^{i,j})\cong V_{\lambda+L}[\hat{g}^i](j)$$ for any $\lambda+L \in L^*/L$.   By \eqref{qform1}, we have 
\[
q(\varphi(\lambda+L, W^{i,j}))= q(V_{\lambda+L}^{\hat{g}} \fusion_{V_{L}^{\hat{g}}} I(W^{i,j})) =q(\lambda+L)+ q(W^{i,j})
\]
as desired. 
\end{proof}

\subsubsection{\textbf{Case 2: }$|\phi_g|=2|g|=2p$}
In this case, $p=|g|$ is even and $2L_{g^{p/2}}^*$ is not doubly even.  Let $\varphi= g^{p/2}$. Then there is an $x\in L_\varphi^*$ such that $(x|x)\in \frac{1}2 +\Z$ and 
\[
E=\{\alpha\in L_\varphi^*\mid (\alpha|\alpha)\in \Z\}\lneq L_\varphi^*.  
\]
For any $x,y\in L_\varphi^*\setminus E$, we have $(x-y|x-y)= (x|x)+(y|y) -2(x|y)\in \Z$.   
Therefore, $[L_\varphi^*:E]=2$ and $L_\varphi^*= E \cup (x+E)$ for some $x\in L_\varphi^*\setminus E$ . 

Since $\phi_g^p$ has order $2$ and  acts trivially on $V_L\otimes V_K$, $\phi_g^p = \sigma_h$ for some $h\in N/2$ and $h\in (L\oplus K)^*$.  
Without loss, we may assume $h\in L^*$. 

\begin{remark}
Note that for any $a\in L^*$, $\sigma_h=\sigma_a$ if and only if $(h| x) =
( a| x) \mod \Z$ for all $x\in L^*$. It implies $h-a\in L$; hence, 
$h+ L$ is uniquely determined. 
\end{remark}

Since $\phi_g^p=\sigma_h$ is an inner automorphism, the irreducible $\phi_g^p$-twisted module $V_N[ \phi_g^p]$ is given by $V_N^{(h)} \cong V_{h+N}$. 
Recall that 
$  
V_N= \bigoplus_{\lambda\in L^*/L}  V_{\lambda+L}\otimes V_{f(\lambda)+K}.
$ 
Therefore, we have 
\[
V_N[ \phi_g^p]= V_{h+N} = \bigoplus_{\lambda\in L^*/L} V_{h+\lambda+L}\otimes V_{f(\lambda)+K}. 
\]
Moreover, for each $0\leq i\leq p-1$, we have 
\[
\begin{split}
V_{N}[\phi_{g}^i] & =\bigoplus_{\lambda\in L^*/L} V_{\lambda+L}[\hat{g}^i]\otimes V_{f(\lambda)+K} , \\
V_{N}[\phi_{g}^{p+i}] & =\bigoplus_{\lambda\in L^*/L} V_{h+\lambda+L}[\hat{g}^i]\otimes V_{f(\lambda)+K} .
\end{split}
\] 
In this case,  we have 
\[
V_N^{\phi_g}= \bigoplus_{\lambda\in X/L}  (V_{\lambda+L}\otimes V_{f(\lambda)+K})^{\phi_g}.
\]
where $X=\{ x\in L^*\mid (h|x) \in \Z\}$. 

\begin{remark}\label{choiceofu}
Note that $[L^*:X]=2$ and $L^*=X\cup (u+X)$ for some $u\in L^*$.   Notice that $u+h\in X$ if $h\notin X$ (i.e., $(h|h)\notin \Z$). In this case, $h+L=u+\lambda+L$ for some $\lambda\in L^*$ and we may assume $u=h$. If $h\in X$ (i.e., $(h|h)\in \Z$), $h+X\neq u+X$ and we may take any $u\in L^*\setminus X$. 
\end{remark}
\medskip

Let $\rho$  be the conformal weight of the unique irreducible $\phi_g$-twisted $V_N$-module $V_N[\phi_g]$ and let  $t\in \{0, 1, \dots, 2p-1\}$ such that $t=4p^2\rho \mod 2p$.  Then 
$$q(W^{i,j}) \equiv \frac{ij}{2p} + \frac{i^2t}{4p^2} \mod \Z. $$ 
Notice that $\rho$ is also the conformal weight of $V_{\lambda+L}[\hat{g}]$ for any $\lambda\in L^*$ and $\hat{g}=\phi_g|_{V_{L}}$ has order $p$. Thus, the weights of $ V_{\lambda+ L}[\hat{g}]\otimes V_{f(\lambda)+K} $ are in $\rho+ \frac{( f(\lambda)|  f(\lambda))}2 + \frac{1}p \Z$. 

\medskip

By \cite{EMS}, $\rho -\frac{t}{4p^2}  \in \frac{1}{2p} \Z$.  There are also two cases:

\noindent \textbf{Case a:}   $\rho -\frac{t}{4p^2}  \in \frac{1}{p} \Z$. In this case, 
the weights of $ V_{L}[\hat{g}]\otimes V_{K} $ are in $\frac{t}{4p^2} + \frac{1}{p} \Z$ and we have    
\begin{equation}\label{dwij2p}
W^{i,j} = 
\begin{cases}
\displaystyle
\bigoplus_{\lambda\in X/L} (V_{\lambda+ L}[\hat{g}]\otimes V_{f(\lambda)+K})(j),
\quad \text{ if $j$ is even and } 0\leq i< p , \\
\displaystyle \bigoplus_{\lambda\in X/L} (V_{\lambda+u+ L}[\hat{g}]\otimes V_{f(u+\lambda)+K})(j), \quad \text{ if $j$ is odd and } 0\leq i< p, \\ 
\displaystyle
\bigoplus_{\lambda\in X/L} (V_{h+\lambda+ L}[\hat{g}]\otimes V_{f(\lambda)+K})(j),
\quad \text{ if $j$ is even and } p\leq i<  2p, \\
\displaystyle \bigoplus_{\lambda\in X/L} (V_{h+u+\lambda+ L}[\hat{g}]\otimes V_{f(u+\lambda)+K})(j), \quad \text{ if $j$ is odd and } p\leq i < 2p.  
\end{cases}
\end{equation}

Since  
$(V_{\lambda+ L}[\hat{g}]\otimes V_{f(\lambda)+K})(j)$ and $(V_{h+\lambda+ L}[\hat{g}]\otimes V_{f(\lambda)+K})(j)$ are irreducible $V_{L}^{\hat{g}}\otimes V_{K} $-modules for any $0 \leq j< 2p$ and $\lambda\in L^*$, there exist $0\leq k, k' < p$ such that 
\[
\begin{split}
&(V_{\lambda+ L}[\hat{g}]\otimes V_{f(\lambda)+K})(j)\cong V_{\lambda+ L}[\hat{g}](k) \otimes V_{f(\lambda)+K}, \quad \text{ and }\\
& (V_{h+\lambda+ L}[\hat{g}]\otimes V_{f(\lambda)+K})(j)\cong V_{h+\lambda+ L}[\hat{g}](k')\otimes V_{f(\lambda)+K}. 
\end{split}
\]

By adjusting the action of $\hat{g}$ on $V_{\lambda+L_g}[\hat{g}^i]$ if necessary, we may assume  
\[
\begin{split}
(V_{\lambda+ L}[\hat{g}^i]\otimes V_{f(\lambda)+K})(j) &\cong  V_{\lambda+ L}[\hat{g}^i](\frac{j}2) \otimes V_{f(\lambda)+K}, \\
& \qquad \text{ if $j$ is even and } 0\leq i< p , \\
(V_{u+\lambda+ L}[\hat{g}^i]\otimes V_{f(u+\lambda)+K})(j)&\cong  V_{u+\lambda+ L}[\hat{g}^i](\frac{j-1}2) \otimes V_{f(u+\lambda)+K},\\
& \qquad \text{ if $j$ is odd and } 0\leq i< p, \\ 
(V_{h+\lambda+ L}[\hat{g}^i]\otimes V_{f(h+\lambda)+K})(j) &\cong V_{h+\lambda+ L}[\hat{g}^{i-p}](\frac{j}2) \otimes V_{f(h+\lambda)+K},\\
& \qquad \text{ if $j$ is even and } p\leq i < 2p, \\
(V_{h+u+\lambda+ L}[\hat{g}^i]\otimes V_{f(h+u+\lambda)+K})(j)&\cong V_{h+u+\lambda+ L}[\hat{g}^{i-p}](\frac{j-1}2) \otimes V_{f(h+u+\lambda)+K},\\
& \qquad \text{ if $j$ is odd and }  p\leq i < 2p. \\  
\end{split}
\]

Define $I: R(V_N^{\phi_g}) \to R(V_{L}^{\hat{g}})$ such that 
\begin{equation}\label{case2a}
I(W^{i,j}) = 
\begin{cases}
V_{u+ L}[\hat{g}^i]( \frac{j-1}2) & \text{ if } 0\leq i< p,\ j \text{ odd}, \\
V_{L}[\hat{g}^i]( \frac{j}2) & \text{ if } 0\leq i< p,\ j \text{ even},\\
V_{h+u+ L}[\hat{g}^{i-p}]( \frac{j-1}2) & \text{ if } p\leq i< 2p,\ j \text{ odd},\\
V_{h+ L}[\hat{g}^{i-p}]( \frac{j}2) & \text{ if } p\leq i< 2p,\  j \text{ even}.
\end{cases}
\end{equation}
Notice that the map $I$ may depend on the choice of $u$. 

\medskip

\noindent \textbf{Case b:} $\rho -\frac{t}{4p^2}  \in \frac{1}{2p} \Z\setminus \frac{1}{p} \Z$.  
In this case, the weights of the module  $V_{L}[\hat{g}]\otimes V_{K}  $ are in $\frac{t}{4p^2} + \frac{1}{2p} \Z\setminus \frac{1}{p} \Z$ but the weights of $V_{u+L}[\hat{g}]\otimes V_{f(u)+K} $ are in $\frac{t}{4p^2} + \frac{1}p \Z$.  Then 
\[
W^{1,0} = \bigoplus_{\lambda\in X/L} (V_{u+\lambda+ L}[\hat{g}]\otimes V_{f(u+\lambda)+K})(0); 
\]
notice that $q(W^{1,0}) = t/4p^2 \mod \Z$. Similarly, we also have 
\[
W^{1,j} = 
\bigoplus_{\lambda\in X/L} (V_{(\bar{j}+1)u+\lambda+ L}[\hat{g}]\otimes V_{f((\bar{j}+1)u+\lambda)+K})(j) 
\]
where 
\begin{equation}\label{bar}
\bar{j}=\begin{cases}
1  & \text{  if $j$ is odd}, \\
0  & \text{ if $j$ is even}.
\end{cases}
\end{equation}

By the fusion rules, we have 
\[
W^{i,0} = 
\begin{cases}
\displaystyle
\bigoplus_{\lambda\in X/L} (V_{\bar{i}u+\lambda+ L}[\hat{g}^i]\otimes V_{f(\bar{i}u+\lambda)+K})(0)  &\text{ if }  0\leq i< p,\\
\displaystyle
\bigoplus_{\lambda\in X/L} (V_{h+\bar{i}u+\lambda+ L}[\hat{g}^{i-p}]\otimes V_{f(\bar{i}u+\lambda)+K})(0)  &\text{ if }  p\leq i< 2p,
\end{cases}
\]
and hence we have 
\begin{equation}\label{dwij2pb}
W^{i,j} = 
\begin{cases}
\displaystyle
\bigoplus_{\lambda\in X/L} (V_{(\bar{i}+\bar{j})u+\lambda+ L}[\hat{g}^i]\otimes V_{f((\bar{i}+\bar{j})u+\lambda)+K})(j)  &\text{ if }  0\leq i< p,\\
\displaystyle
\bigoplus_{\lambda\in X/L} (V_{h+(\bar{i}+\bar{j})u+\lambda+ L}[\hat{g}^{i-p}]\otimes V_{f((\bar{i}+\bar{j})u+\lambda)+K})(j)  &\text{ if }  p\leq i< 2p,
\end{cases}
\end{equation}

By adjusting the action of $\hat{g}$ on $V_{\lambda+L_g}[\hat{g}^i]$, we may also assume  
\[
\begin{split}
(V_{\lambda+ L}[\hat{g}^i]\otimes V_{f(\lambda)+K})(j) &\cong  V_{\lambda+ L}[\hat{g}^i](\left\lfloor\frac{j}2\right\rfloor) \otimes V_{f(\lambda)+K}, \\
& \qquad \text{ if $i+j$ is even and } 0\leq i< p , \\
(V_{u+\lambda+ L}[\hat{g}^i]\otimes V_{f(u+\lambda)+K})(j)&\cong  V_{u+\lambda+ L}[\hat{g}^i](\left\lfloor\frac{j}2\right\rfloor) \otimes V_{f(u+\lambda)+K},\\
& \qquad \text{ if $i+j$ is odd and } 0\leq i< p, \\ 
(V_{h+\lambda+ L}[\hat{g}^i]\otimes V_{f(h+\lambda)+K})(j) &\cong V_{h+\lambda+ L}[\hat{g}^{i-p}](\left\lfloor\frac{j}2\right\rfloor) \otimes V_{f(h+\lambda)+K},\\
& \qquad \text{ if $i+j$ is even and } p\leq i < 2p, \\
(V_{h+u+\lambda+ L}[\hat{g}^i]\otimes V_{f(h+u+\lambda)+K})(j)&\cong V_{h+u+\lambda+ L}[\hat{g}^{i-p}](\left\lfloor\frac{j}2\right\rfloor) \otimes V_{f(h+u+\lambda)+K},\\
& \qquad \text{ if $i+j$ is odd and }  p\leq i < 2p, \\  
\end{split}
\]  
where $\lfloor x\rfloor$ denotes the greatest integer that is less than or equal to $x$.

Define $I: R(V_N^{\phi_g}) \to R(V_{L}^{\hat{g}})$ such that 
\begin{equation}\label{case2b}
I(W^{i,j}) = 
\begin{cases}
V_{u+ L}[\hat{g}^i](\left\lfloor\frac{j}2\right\rfloor) & \text{ if } 0\leq i< p,\ i+j \text{ odd}, \\
V_{L}[\hat{g}^i](\left\lfloor\frac{j}2\right\rfloor) & \text{ if } 0\leq i< p,\ i+j \text{ even},\\
V_{h+u+ L}[\hat{g}^{i-p}](\left\lfloor\frac{j}2\right\rfloor) & \text{ if } p\leq i< 2p,\ i+j \text{ odd},\\
V_{h+ L}[\hat{g}^{i-p}](\left\lfloor\frac{j}2\right\rfloor) & \text{ if } p\leq i< 2p,\  i+j \text{ even}.
\end{cases}
\end{equation}
Again the map $I$ may depend on the choice of $u$. 


\begin{lemma}\label{mapI}
Let $I$ be defined as in \eqref{case2a} or \eqref{case2b}. 
The map $I$ is a group homomorphism. Moreover, $I$ is $1$ to $1$ if $(h|h)\in \Z$, i.e., $h\in X$; otherwise, $I$ is $2$ to $1$. 
\end{lemma}

\begin{proof}
It follows from formulas \eqref{FWij}, \eqref{dwij2p} and \eqref{dwij2pb}. 
\end{proof}

Let $Y=\{ a\in L^*\mid (a| h) \in \Z\text{ and } (a| u)\in \Z \}$. Then $X> Y > L$. Note that $h, u,\notin Y$ since $( h| u) \notin \Z$, and  $(Y/L) \times  H \cong L^*/ L$,   
where $H$ is  the subgroup of $L^*/ L$ generated by $h+L$ and $u+L$.   Recall that we take $u=h$ if $h\notin X$, i.e., $( h| h) \notin \Z$. In this case, $X=Y$; otherwise, we have $[L^*: Y]= 2^2$. 

For any $0\leq i,j<2p$, we define 
\begin{equation}\label{ep}
\varepsilon_{i,j}=
\begin{cases}
\bar{j} &\text{ if } \rho -\frac{t}{4p^2}  \in \frac{1}{p} \Z, \\
\overline{i + j} & \text{ if } \rho -\frac{t}{4p^2}  \notin \frac{1}{p} \Z. 
\end{cases}
\end{equation}

\begin{theorem}\label{Rpg=2n}
As an abelian group,  we have
\[
R(V_{L}^{\hat{g}})\cong Y/L \times I(R(V_N^{\phi_g})).
\]
\end{theorem} 

\begin{proof}
Recall that 
\[
\begin{split}
V_N^{\phi_g} = \bigoplus_{\lambda\in X/L} (V_{\lambda+L} \otimes V_{f(\lambda)+K})^{\phi_g}=  \bigoplus_{\lambda\in X/L} V_{\lambda+L}^{\hat{g}} \otimes V_{f(\lambda)+K}.  
\end{split}
\]
Then as a module for $V_{L}^{\hat{g}}\otimes V_{K} $-module, 
\begin{equation}\label{dWij2p1}
\begin{split}
W^{i,j} &= 
V_N^{\phi_g} \fusion_{ V_{L}^{\hat{g}}\otimes V_{K}} (  I(W^{i,j})\otimes V_{f(\varepsilon_{i,j}u) +K} ), \\
&= \bigoplus_{\lambda\in X/L}  \left(V_{\lambda+L}^{\hat{g}}\fusion_{V_{L}^{\hat{g}}} I(W^{i,j})\right)\otimes V_{f(\varepsilon_{i,j}u+\lambda)+ K} .  
\end{split}
\end{equation}

By the decomposition of $W^{i,j}$ as irreducible $V_{L}^{\hat{g}}\otimes V_{K} $-modules (cf. \eqref{dwij2p} and \eqref{dwij2pb}), we have 
\begin{equation}\label{dWij2p2}
V_{\lambda+L}^{\hat{g}} \fusion_{V_{L}^{\hat{g}}} I(W^{i,j})=
\begin{cases}
V_{\lambda+u+ L}[\hat{g}^i](\left\lfloor\frac{j}2\right\rfloor) & \text{ if } 0\leq i< p,\ \varepsilon_{i,j}=1, \\
V_{\lambda+L}[\hat{g}^i](\left\lfloor\frac{j}2\right\rfloor) & \text{ if } 0\leq i< p,\ \varepsilon_{i,j}=0,\\
V_{\lambda+h+u+ L}[\hat{g}^{i-p}](\left\lfloor\frac{j}2\right\rfloor) & \text{ if } p\leq i< 2p,\ \varepsilon_{i,j}=1,\\
V_{\lambda+h+ L}[\hat{g}^{i-p}](\left\lfloor\frac{j}2\right\rfloor) & \text{ if } p\leq i< 2p,\  \varepsilon_{i,j}=0.
\end{cases}
\end{equation}

Now define $\varphi:  Y/L \times I(R(V_N^{\phi_g}))\to R(V_{L}^{\hat{g}})$ by 
\[
\varphi( \lambda+L, I(W^{i,j})) =  V_{\lambda+L}^{\hat{g}} \fusion_{V_{L}^{\hat{g}}} I(W^{i,j}). 
\]

It follows from the fusion rules of $V_{K}$ and $V_N^{\phi_g}$ that $\varphi$ is a group homomorphism. Moreover, it is injective by Lemma \ref{mapI} and the fact that $h+L$ and $u+L$ are orthogonal to $ Y/L$ with respect to the standard bilinear form. 
\end{proof}

For the quadratic form, we note that 
\begin{equation}\label{qwij}
q(W^{i,j}) = 
\begin{cases}
q(V_{L}[\hat{g}^i](\left\lfloor\frac{j}2\right\rfloor)),
\quad \text{ if } \varepsilon_{i,j}=0\text{ and } 0\leq i< p , \\
q(V_{f(u)+K})+ q(V_{u+L}[\hat{g}^i](\left\lfloor\frac{j}2\right\rfloor)), \quad \text{ if } \varepsilon_{i,j}=1\text{ and } 
0\leq i< p, \\ 
q( V_{h+L}[\hat{g}^{i-p}](\left\lfloor\frac{j}2\right\rfloor)),
\quad \text{ if } \varepsilon_{i,j}=0\text{ and } p\leq i<  2p, \\
q(V_{f(u)+K})+ q(V_{h+u+L}[\hat{g}^{i-p}](\left\lfloor\frac{j}2\right\rfloor)), \quad \text{ if } \varepsilon_{i,j}=1\text{ and } p\leq i < 2p,
\end{cases}
\end{equation}
by \eqref{dwij2p} and \eqref{dwij2pb}. Hence, 
\begin{equation}\label{qIw}
q(I(W^{i,j})) = 
\begin{cases}
\frac{ij}{2p}+ \frac{i^2 t}{4p^2} \mod \Z  & \text{ if } \varepsilon_{i,j}=0,  \\
\frac{ij}{2p}+ \frac{i^2 t}{4p^2} +\frac{( u| u)}2 \mod \Z  & \text{ if  } \varepsilon_{i,j}=1. 
\end{cases}
\end{equation}

For any $\lambda+L \in Y$, we have $\varphi(\lambda+L, I(W^{i,j})) =V_{\lambda+L}^{\hat{g}} \fusion_{V_{L}^{\hat{g}}} I(W^{i,j})$. By \eqref{dWij2p1},  the conformal weight of $\varphi(\lambda+L, I(W^{i,j}))$ is given by 
\[
-\frac{( f(\lambda)| f(\lambda) )}2 + q(I(W^{i,j}))\equiv \frac{( \lambda| \lambda )}2 + q(I(W^{i,j}))\quad \mod \Z. 
\]

\begin{theorem}\label{thm:2a}
Suppose $|\phi_g|=2|g|=2p$. Then 
\[
q(\varphi(\lambda+L, I(W^{i,j})) = 
\begin{cases}
\frac{ij}{2p}+ \frac{i^2 t}{4p^2} +\frac{( \lambda| \lambda)}2 \mod \Z  & \text{ if  } \varepsilon_{i,j}=0, \\
\frac{ij}{2p}+ \frac{i^2 t}{4p^2} +\frac{(u|  u)}2 + \frac{( \lambda| \lambda)}2\mod \Z  & \text{ if } \varepsilon_{i,j}=1. 
\end{cases}
\]
Moreover,  $(R(V_L^{\hat{g}}), q) \cong (Y/L,q) \times (I(R(V_N^{\phi_g})), q)$. 
\end{theorem}

\section{Leech lattice and some explicit examples}\label{sec:6}
Next we will study several explicit examples associated with the Leech lattice. 
In \cite{Ho2}, a construction of holomorphic vertex operator algebras of central charge 24 has been proposed by using some orbifold VOAs associated with coinvariant lattices of the Leech lattice (see \cite[Table 4 and Tables 5--15]{Ho2}). In particular, H\"ohn conjectured that $(R(V_{\Lambda_g}^{\hat{g}}),q)\cong (R( V_{L_\mathfrak{g}}), -q)$ as quadratic spaces for some lattice $L_\g$.  The corresponding isometries  in $O(\Lambda)$ and the related lattices are given in Table \ref{T:g} and Table \ref{T:Lg}.  See Section \ref{intro} for the definition of $Q_\g$ and $L_\g$. In this section, we will verify his conjecture and prove the following theorem. Recall that the cases for prime order elements have been treated in \cite{LS17} (see also \cite{HS}) and thus we only consider the cases when $|g|$ is not a prime in this section.  

\begin{theorem}\label{Th2}
Let $g\in O(\Lambda)$ be as in Table \ref{T:g} and $L_\g$ the lattice associated with $g$ as stated in Table \ref{T:Lg}. Then, 
\[
(R(V_{\Lambda_g}^{\hat{g}}),q)\cong (R( V_{L_\mathfrak{g}}), -q)\cong (\mathcal{D}(L_\g), -q)
\]
as quadratic spaces. 
\end{theorem}

\begin{longtable}{|c|c|c|c|c|c|}
\caption{Standard lift of $g\in O(\Lambda)$}\label{T:g}
\\ \hline 
$|g|$ & $rank(\Lambda^g)$& Conjugacy class& Cycle shape& $|\phi_g|$& Conformal weight of $V_\Lambda[\phi_g], \rho$ \\ \hline
$4$&$10$&$4C$& $1^42^24^4$ &$4$&$3/4$\\
$6$&$6$&$6G$&$2^36^3$&$12$&$11/12$\\
$6$& $8$& $6E$ & $1^22^23^26^2$&$6$& $5/6$ \\
$8$&$6$&$8E$&$1^22^14^18^2$&$8$&$7/8$\\
$10$&$4$&$10F$&$2^2 {10}^2$&$20$&$19/20$\\ 
\hline 
\end{longtable}

\begin{longtable}{|c|c|c|c|}
\caption{The lattice $L_\g$}\label{T:Lg}
\\ \hline 
Conjugacy class&   Root system $Q_\g$& $[L_\g: Q_\g]$& $\mathcal{D}(L_\g)$ \\ \hline
$4C$& $2E_6A_2A_1^2$ &$3$&$2^2\cdot 4^6$\\
$6G$&$\sqrt{6}D_4\sqrt{2}A_2$&$1$&$2^4\cdot 4^2\cdot 5^3$\\
$6E$ & $\sqrt{3}A_1^5\sqrt{2}A_2 A_1$&$2$& $2^6\cdot 3^6$ \\
$8E$&$\sqrt{8}D_5^*\sqrt{2}A_1$&$1$&$2\cdot 4\cdot 8^4$\\
$10F$&$\sqrt{10}D_4$& $1$ & $2^2\cdot 4^2\cdot 5^4$\\ 
\hline 
\end{longtable}

By Theorems \ref{thm:5.2} and \ref{thm:2a}, we know that
\[
(R(V_{\Lambda_g}^{\hat{g}}), q) \cong (\mathcal{D}(\Lambda_g), q) \times (R(V_\Lambda^{\phi_g}), q)\cong (\mathcal{D}(\Lambda^g), -q) \times (R(V_\Lambda^{\phi_g}), q)  
\]
if  $|\phi_g|=|g|$ and 
\[
(R(V_{\Lambda_g}^{\hat{g}}), q) \cong (Y/\Lambda_g,q) \times (I(R(V_\Lambda^{\phi_g})), q)
\]
if $|\phi_g|=2|g|$, where $Y=\{ a\in (\Lambda_g)^*\mid ( a| h)\in \Z\text{ and } (a| u)\in \Z \}$.  The main idea is to find a subgroup $H< \mathcal{D}(L_\g)$ such that $(H, -q|_H) \cong ( I(V_{\Lambda}^{\phi_g}), q)$ and $(H^\perp, -q|_{H^\perp})\cong (Y/\Lambda_g, q)$, where  $H^\perp$ is the subgroup of $\mathcal{D}(L_\g)$ orthogonal to $H$. Basically, all calculations are about the structures of the lattices $L_\g$ and $\Lambda_g$ (or $\Lambda^g$).  
  
By Magma, it is easy to verify that $\Lambda_g= (1-g)\Lambda= (1-g) \Lambda_g^*$ 
for an isometry $g$ listed in Table \ref{T:g}. Hence, $|\Lambda_g^*/ \Lambda_g|=|\det(1-g)|$ by Lemma \ref{detLg}. In fact, $\mathcal{D}(\Lambda_g)$ is determined by the cycle shape of $g$ for these cases. The discriminant groups and the corresponding quadratic structures for the fixed point lattice $\Lambda^g$ and the lattice $L_\g$ can also be computed by Magma (see \cite{HL90} and \cite{HM} for explicit information about the fixed point lattices).

\medskip

First, we discuss the cases which $|\phi_g|=|g|$, i.e., $4C, 6E$ and $8E$.

\subsection{Conjugacy class $4C$}
Let $g$ be an isometry of conjugacy class $4C$ in $O(\Lambda)$. Then  $g^2$ is in the conjugacy class $2A$ and $\Lambda_{g^2}\cong \sqrt{2}E_8$, which is doubly even.  In this case, the fixed point sublattice of $g$ has rank $10$ and  
$\mathcal{D}(\Lambda^g)\cong \mathcal{D}(\Lambda_g) \cong 2^{2}\times 4^{4}$ \cite{HL90}. 
 
Since $V_\Lambda[\phi_g]$ has the conformal weight $3/4$, we have $t\equiv 0 \mod 4$; thus,
$R(V_\Lambda^{\phi_g})\cong 4^2$. 
By Theorem \ref{thm:5.2},  
\[
R(V_{\Lambda_g}^{\hat{g}})\cong \mathcal{D}(\Lambda_g) \times R(V_\Lambda^{\phi_g}) \cong 2^2\times 4^6
\]
and 
\[
(R(V_{\Lambda_g}^{\hat{g}}), q) \cong (\mathcal{D}(\Lambda_g), q) \times (R(V_\Lambda^{\phi_g}), q)\cong (\mathcal{D}(\Lambda^g), -q) \times (R(V_\Lambda^{\phi_g}), q). 
\]

When $g$ is a $4C$ element of $O(\Lambda)$, $L_\g$ is an index $3$ of $2E_6 A_2A_1^2$ with a glue vector $v=2\gamma+ \eta$, where 
$\gamma\in E_6^*$  with norm $4/3$ and $\eta\in A_2^*$ with norm $2/3$ such that $\gamma+E_6$ generates $E_6^*/E_6$ and 
$\eta+A_2$ generates $A_2^*/A_2$. 

Let $\{\alpha_1, \cdots, \alpha_6\}$ be a set of simple roots for $E_6$ such that $\mathrm{Span}_\Z\{\alpha_1, \cdots, \alpha_5\}\cong A_5$ and $(\alpha_3|\alpha_6)=-1$.  We also use $\alpha_0$ to denote the negative of the highest root. 
We also let  $\{\beta_1, \beta_2\}$  be a set of simple roots for $A_1^2$.
Then $\mathcal{D}(L_\g)$ has a set of generator $\{{a}_1, \cdots, {a}_6, {b}_1,{b}_2\}$, where ${a}_i =\alpha_i/2 + L_\g, i=0,1,\dots,6$ and ${b}_i=\beta_i/2 +L_\g, i=1,2$. 

Notice that ${a}_1, \cdots, {a}_6$ have order $4$ and ${b}_1, {b}_2$ have order $2$ in $\mathcal{D}(L_\g)$. Moreover, 
\[
q({a}_1)= \cdots = q({a}_6) = q({b}_1)= q({b}_2)=\frac{1}4 \mod \Z 
\]
and ${a}_1, \cdots, {a}_6$ are orthogonal to ${b}_1, {b}_2$.

By Lemma \ref{mapIp} and Theorem \ref{thm:5.2}, $I(W^{i,j}) =V_{\Lambda_g}[\hat{g}^i](j)$ and 
$q( I(W^{i,j})) = ij/4 \mod \Z$.

Let $x_1= ({a}_0+{a}_6) +{a}_1- \bar{a}_2$ and $x_2= {a}_4-{a}_5+{a}_6$. Then 
\[
q(x_1)=q(x_2)=0 \quad \text{ and } \quad (x_1|x_2) = -1/4 \mod \Z.
\]
Denote the subgroup of $\mathcal{D}(L_\g)$ generated by $x_1$ and $x_2$ by $H$.  
It is straightforward to show that $H\cong 4^2$ and $(H, -q|_H) \cong ( I(V_{\Lambda}^{\phi_g}), q)$. Moreover, the subgroup $H^\perp$ of $\mathcal{D}(L_\g)$ orthogonal to $H$ is generated  by
\[
y_1= {a}_1+{a}_2,\  y_2={a}_3 ,\   y_3={a}_4+{a}_5,\  y_4={a}_1-{a}_5 -({a}_0+2{a}_6) , \  
{b}_1,\  \bar{b}_2.  
\]
Note that $q(y_1)=q(y_2)=q(y_3)=q(y_4)=1/4 \mod \Z$.

By Magma, it is straightforward to verify that $(\Lambda^g)^*$  has the minimal norm $3/2$ and $\mathcal{D}(\Lambda_g^*)$ contains a set of generator $\{a_1', a_2',a_3',a_4', b_1', b_2'\}$ such  that $ a_1', a_2',a_3',a_4'$  have  order $4$ and 
$b_1', b_2'$ have order $2$. Moreover, $q(a_1')=q(a_2')=q(a_3')=q(a_4')=q(b_1')=q(b_2')=1/4 \mod \Z$ and $ (H^\perp, q)\cong (\mathcal{D}(\Lambda_g), q)$.

\subsection{Conjugacy class $6E$}
Let $g$ be an isometry of conjugacy class $6E$ in $O(\Lambda)$. Then 
$g^2$ is in the conjugacy class $3A$ and $g^3$ is in the conjugacy class $2A$. 
The fixed point sublattice of $g$ has rank $8$ and is isometric to $A_2\otimes D_4$. Moreover, the discriminant form $(\mathcal{D}(\Lambda^g), q) \cong 2^{4, +}\times 3^{4, +}$ (cf. \cite{GL3A6A}).  Since the conformal weight of $V_{\Lambda}[\phi_g]$ is $5/6$, we have $t\equiv 0 \mod 6$ and thus $R(V_{\Lambda}^{\phi_g})\cong 6^2$. By \eqref{qij}, it is easy to verify that the quadratic form associated with 
$(R(V_{\Lambda}^{\phi_g}), q)$ is isometric to $2^{2,+}\times 3^{2,+}$. Thus, $R(V_{\Lambda_g}^{\hat{g}})$ is isometric to a quadratic space of type $2^{6,+}\times 3^{6, +}$. 

Recall that $L_\g$ is an index $2$ overlattice of $\sqrt{3}A_1^5\sqrt{2}A_2 A_1$.   By a direct calculation, it is straightforward to verify that $(\mathcal{D}(L_\g), -q)$ is a quadratic space of type $2^{6,+}\times 3^{6, +}$. 

\subsection{Conjugacy class $8E$}
Let $g$ be an isometry of conjugacy class $8E$ in $O(\Lambda)$. Then $g^2$ is in the conjugacy class $4C$ and $g^4$ is in the class $2A$.
The fixed point sublattice of $g$ has rank $6$ and $\mathcal{D}(\Lambda^g)\cong 2\times 4 \times 8^2$. By using computer, one can show that there is  a set of generators $\{x_1, y_1,z_1, z_2\}$ of $\mathcal{D}(\Lambda^g)$ such that $|x_1|=2, |y_1|=4$, $|z_i|=8$ for $i=1,2$; $q(x_1)=1/4$, $q(y_1)=1/8$, $q(z_i)=1/8 \mod \Z, i=1,2$; $\{x_1\}\perp \{y_1\}\perp \{z_1, z_2\} $  and $(z_1|z_2) = -1/8 \mod \Z$.

For $g=8E$, we can choose $L_\g= \sqrt{8}D_5^*\sqrt{2}A_1$. Then $\mathcal{D}(L_\g)\cong 2\cdot 4\cdot 8^4$ and there is a
set of generators $\{a_1, b_1, e_1,e_2, e_3, e_4\}$ of $\mathcal{D}(L_\g)$ such that $|a_1|=2, |b_1|=4$, $|e_i|=8$ for $i=1,2,3,4$; $q(a_1)=1/4$, $q(b_1)=1/8$, $q(e_i)=1/8\mod \Z, i=1,2,3,4$; $(e_i|e_{i+1}) =-1/8 \mod \Z, i=1,2,3$, and $\{a_1\}\perp \{b_1\}\perp \{e_1, e_2,e_3,e_4\} $.

Now it is straightforward to show that $(\mathcal{D}(\Lambda^g), -q) \times (R(V_\Lambda^{\phi_g}), q)\cong (\mathcal{D}(L_\g), -q)$. 
\medskip

Next we discuss the cases that $|\phi_g|= 2|g|$. 

\subsection{Conjugacy class $6G$}
Let $g$ be an isometry of conjugacy class $6G$ in $O(\Lambda)$. Then $g^2$ is in the conjugacy class $3B$ and $g^3$ is in the class $2C$. In this case, $\Lambda_{g^3}\cong \sqrt{2}D_{12}^+$, which is not doubly even.  Therefore, $|\phi_g|$ has order $12$ and $\phi_g^6=\sigma_h$ for some $h\in\frac{1}2 \Lambda_g$ and $( h| h)=2$. We can also choose $u$ such that $(u| u)= \frac{3}2$. The irreducible $\phi_g$-twisted module has the conformal weight $\rho=11/12$; hence we have $t\equiv 0\mod 12$ and $\rho \notin \frac{1}6\Z$. That means we have Case 2b as described in Section \ref{S:5}. 
By Theorem \ref{thm:2a}, we have 
\[
q(I(W^{i,j})) = \frac{ij}{12} + \frac{3}4(\overline{i+j}) \mod \Z. 
\]
In particular, $q(W^{0,1}) =q(W^{1,0}) = 3/4 =-1/4$ and $q(W^{1,1})=1/12 \mod \Z$.

In this case,  the fixed point lattice $\Lambda^g$ has rank $6$ and is isometric to $\sqrt{2}(A_3+\sqrt{3}A_3)^+$. Moreover, $\mathcal{D}(\Lambda^g)\cong 2^6\times 3^3$ and the quadratic form on $\mathcal{D}(\Lambda^g)$ takes values in $\frac{1}{6} \Z$. By computer, it is easy to verify that 
$(\mathcal{D}(\Lambda_g) \cong (\mathcal{D}(\Lambda^g),-q)$ is a non-singular quadratic space of type $2^{6,+} \times 3^{3}$.   

Since $(h| h) \in \Z$, $I$ is injective and we have $R(V_{\Lambda_g}^{\hat{g}})\cong 2^4 \times 4^2 \times 3^5$ 
as an abelian group by Theorem \ref{Rpg=2n}. 
It is also  easy to verify that 
$(Y/\Lambda_g, q) \cong 2^{4,+}\times 3^3$ as a quadratic space, where $Y=\{ a\in (\Lambda_g)^*\mid ( a| h)\in \Z\text{ and } (a| u)\in \Z \}$. 

For $g\in 6G$, $L_\g\cong \sqrt{6}D_4 \sqrt{2}A_2$ and $\mathcal{D}(L_\g)\cong 6^2\cdot {12}^2\cdot 2^2\cdot 3\cong 2^4\cdot 4^2\cdot 3^5$ as an abelian group.  By a direct calculation, it is straightforward to find a subgroup $H$ of $\mathcal{D}(L_\g)$ such that $(H, -q|_H) \cong (I(R(V_\Lambda^{\phi_g})), q)$ and to show that  $(Y/\Lambda_g), q) \times (I(R(V_\Lambda^{\phi_g})), q)\cong (\mathcal{D}(L_\g), -q)$.

\subsection{Conjugacy class $10F$}
Let $g$ be an isometry of conjugacy class $10F$ in $O(\Lambda)$. Then $g^2$ is in the conjugacy class $5B$ and $g^5$ is in the class $2C$. The fixed point sublattice has rank $4$ and the discriminant form $(\mathcal{D}(\Lambda^g), q)$ is isometric to $2^{4,+} \times 5^{2,+}$. 
In this case, $|\phi_g|$ has order $20$ and $\phi_g^{10}=\sigma_h$ for some $h\in\frac{1}2 \Lambda_g$ and $( h| h)=2$. We can choose $u$ such that $(u| u) =3/2$. Since $\rho=19/20\notin 1/10\Z$, we have Case 2b as described in Section \ref{S:5}. In this case, 
 $t\equiv 0\mod 20$ and $( h| h)\in \Z$; thus, we have $R(V_{\Lambda_g}^{\hat{g}})\cong 2^2 \times 4^2 \times 5^4$ 
as an abelian group by Theorem \ref{Rpg=2n}. 

Let $Y=\{ a\in (\Lambda_g)^*\mid ( a| h)\in \Z\text{ and } (a| u) \in \Z \}$. Then it can be verified that 
$(Y/\Lambda^g, q) \cong 2^{2,+} \times 5^{2,+}$. By Theorem \ref{thm:2a}, we also have  
\[
q(I(W^{i,j})) = \frac{ij}{20} +\frac{3}4(\overline{i+j}) \mod \Z. 
\]
Again, by direct calculations, one can verify that $(Y/\Lambda_g), q) \times (I(R(V_\Lambda^{\phi_g})), q)\cong (\mathcal{D}(L_\g), -q)$ as desired.

\section{Reconstruction based on simple current extensions}\label{sec:7} To illustrate  H\"ohn's idea \cite{Ho2}, we will describe a new construction of a holomorphic VOA of central charge $24$ such that its weight one Lie algebra has the type $F_{4,6}A_{2,2}$. 

By \cite[Table 4 and 14]{Ho2},  the  lattice $L_\mathfrak{g}$ is isometric to $\sqrt{6}D_4 +\sqrt{2}A_2$ for $\mathfrak{g}\cong F_{4,6}A_{2,2}$ or $D_{4,12}A_{2,6}$. By the discussion in the previous section, we also know that $(R(V_{\Lambda_g}^{\hat{g}}), q)\cong  (\mathcal{D}(L_\g), -q) $ for $g\in O(\Lambda)$ of conjugacy class $6G$. It suggests that a  holomorphic VOA with weight one Lie algebra of type $F_{4,6}A_{2,2}$  or $D_{4,12}A_{2,6}$ can be constructed as a simple current extension of  $V_{\sqrt{6}D_4 +\sqrt{2}A_2} \otimes V_{\Lambda_g}^{\hat{g}}$, where $g\in O(\Lambda)$ is in the conjugacy class $6G$. 

First we recall  a construction of a holomorphic VOA whose weight one Lie algebra has the type $D_{4,12}A_{2,6}$ from \cite{LS6}. 

Let $g\in O(\Lambda)$ be of the class $6G$. Let $\nu\in \frac{1}6 (\Lambda^g)^*\setminus 
\frac{1}6 \Lambda^g$ such that $$( \nu| \nu)=\frac{1}6, \qquad  
|(\nu + (\Lambda^g)^*)(\frac{1}6)|=5$$ and $\varphi=\sigma_\nu \phi_g$ has order $6$ on $V_\Lambda$. For example, we may choose $g=\varepsilon_D \cdot s$, where 
\begin{center}
  \begin{tikzpicture} 
\draw[step=1cm] (0,0) grid (3,2);
\fill (0.25,0.25) circle(1.5pt);
\fill(0.75,0.25) circle(1.5pt);
\fill(1.25,0.25) circle(1.5pt);
\fill(1.75,0.25) circle(1.5pt);
\fill(2.25,0.25) circle(1.5pt);
\fill(2.75,0.25) circle(1.5pt);
\fill(0.25,0.75) circle(1.5pt);
\fill(0.75,0.75) circle(1.5pt);
\fill(1.25,0.75) circle(1.5pt);
\fill(1.75,0.75) circle(1.5pt);
\fill(2.25,0.75) circle(1.5pt);
\fill(2.75,0.75) circle(1.5pt);
\fill(0.25,1.25) circle(1.5pt);
\fill(0.75,1.25) circle(1.5pt);
\fill(1.25,1.25) circle(1.5pt);
\fill(1.75,1.25) circle(1.5pt);
\fill(2.25,1.25) circle(1.5pt);
\fill(2.75,1.25) circle(1.5pt);
\fill(0.25,1.75) circle(1.5pt);
\fill(0.75,1.75) circle(1.5pt);
\fill(1.25,1.75) circle(1.5pt);
\fill(1.75,1.75) circle(1.5pt);
\fill(2.25,1.75) circle(1.5pt);
\fill(2.75,1.75) circle(1.5pt);
\draw[thick, ->] (0.25, 1.25) --  (0.25,0.8); 
\draw[thick, ->] (0.25, 0.75) --  (0.25,0.3);
\draw[thick, ->] (0.75, 1.25) --  (0.75,0.8); 
\draw[thick, ->] (0.75, 0.75) --  (0.75,0.3);
\draw[thick, ->] (1.25, 1.25) --  (1.25,0.8); 
\draw[thick, ->] (1.25, 0.75) --  (1.25,0.3);
\draw[thick, ->] (1.75, 1.25) --  (1.75,0.8); 
\draw[thick, ->] (1.75, 0.75) --  (1.75,0.3);
\draw[thick, ->] (2.25, 1.25) --  (2.25,0.8); 
\draw[thick, ->] (2.25, 0.75) --  (2.25,0.3);
\draw[thick, ->] (2.75, 1.25) --  (2.75,0.8); 
\draw[thick, ->] (2.75, 0.75) --  (2.75,0.3);
\draw (3.1,0.9) node{, }; 
\draw (-0.5,1) node{$s=$};
\end{tikzpicture}
\end{center}
and $\varepsilon_D$ is the involution associated with the dodecade

\begin{center}
  \begin{tikzpicture} 
\draw[step=1cm] (0,0) grid (3,2);
\draw (0.25,0.23) node{*};
\draw (0.75,0.23) node{*};
\draw (1.25,0.23) node{*};
\draw (0.25,0.73) node{*};
\draw (0.75,0.73) node{*};
\draw (1.25,0.73) node{*};
\draw (0.25,1.23) node{*};
\draw (0.75,1.23) node{*};
\draw (1.25,1.23) node{*};
\draw(1.75,1.73) node{*};
\draw(2.25,1.73) node{*};
\draw(2.75,1.73) node{*};
\draw (3.1,0.9) node{. }; 
\draw (-0.5,1) node{$D=$};
\end{tikzpicture}
\end{center}
Moreover, we may choose 
\[
\nu = \tiny \frac{a}{3}\,
\begin{array}{|rr|rr|rr|}
 \hline 2  & 2& 2  & 0 & 0 & 0 \\
 0 &  0 &  0&  0& 0& 0\\
 0 &  0& 0& 0&  0 & 0 \\
 0 &  0& 0&  0& 0 & 0\\ \hline
\end{array},
\]
where $a=1/\sqrt{8}$. Here we use the notion of hexacode balance to
denote the codewords of the Golay code and the vectors in the Leech lattice (see \cite{CS} and \cite{Gr12} for details).

In this case, the VOA obtained by an orbifold construction from $V_\Lambda$ and $\varphi$ has the weight one Lie algebra of the type $D_{4,12}A_{2,6}$ \cite{LS6}.
By a direct calculation, it is easy to verify that 
\[
\begin{split}
&X_1=(\nu + (\Lambda^g)^*)(\frac{1}6)\\ 
= 
& \tiny\left\{ 
\frac{a}{3}\,
\begin{array}{|rr|rr|rr|}
 \hline 2  & 2& 2  & 0 & 0 & 0 \\
 0 &  0 &  0&  0& 0& 0\\
 0 &  0& 0& 0&  0 & 0 \\
 0 &  0& 0&  0& 0 & 0\\ \hline
\end{array},\quad 
\frac{a}{3}\,
\begin{array}{|rr|rr|rr|}
 \hline -1  & -1& -1  & 0 & 0 & 0 \\
 0 &  0 &  0&  1& 1& 1\\
 0 &  0& 0& 1&  1 & 1 \\
 0 &  0& 0&  1& 1 & 1\\ \hline
\end{array}, 
\quad 
\frac{a}{3}\,
\begin{array}{|rr|rr|rr|}
 \hline -1  & -1& -1  & 0 & 0 & 0 \\
 0 &  0 &  0&  -1& -1& 1\\
 0 &  0& 0& -1&  -1 & 1 \\
 0 &  0& 0&  -1& -1 & 1\\ \hline
\end{array},\right. \\
& \left. \tiny
\quad \frac{a}{3}\,
\begin{array}{|rr|rr|rr|}
 \hline -1  & -1& -1  & 0 & 0 & 0 \\
 0 &  0 &  0&  1& -1& -1\\
 0 &  0& 0& 1&  -1 & -1 \\
 0 &  0& 0&  1& -1 & -1\\ \hline
\end{array},\quad 
\frac{a}{3}\,
\begin{array}{|rr|rr|rr|}
 \hline -1  & -1& -1  & 0 & 0 & 0 \\
 0 &  0 &  0&  -1& 1& -1\\
 0 &  0& 0& -1&  1 & -1 \\
 0 &  0& 0&  -1& 1 & -1\\ \hline
\end{array}
\right\}.
\end{split}
\]

Since $V_\Lambda[\varphi^2]$ also has the conformal weight one, $\varphi^2=\sigma_\mu \phi_{g^2}$ for some $\mu \in (\Lambda^{g^2})^*$ and $(\mu|  \mu)=2/3$. Consider the set
\[
X_2=\{ x\in (\mu + (\Lambda^{g^2})^*)(\frac{2}3) \mid P_0^g(x) \in 2\mu + (\Lambda^{g})^*\}.
\]
Then $|X_2|=7$. In fact, $P_0(X_2)$ has $4$ vectors of norm $1/6$ and $3$ vectors of norm $1/3$;  
\[
\begin{split}
&P_0^g(X_2)(\frac{1}6) \\ 
= 
& \tiny\left\{ 
\frac{a}{3}\,
\begin{array}{|rr|rr|rr|}
 \hline 1  & 1& 1  & 0 & 0 & 0 \\
 0 &  0 &  0&  1& 1& 1\\
 0 &  0& 0& 1&  1 & 1 \\
 0 &  0& 0&  1& 1 & 1\\ \hline
\end{array}, 
\quad 
\frac{a}{3}\,
\begin{array}{|rr|rr|rr|}
 \hline 1  & 1& 1  & 0 & 0 & 0 \\
 0 &  0 &  0&  -1& -1& 1\\
 0 &  0& 0& -1&  -1 & 1 \\
 0 &  0& 0&  -1& -1 & 1\\ \hline
\end{array},\right. \\
& \left. \tiny
\quad \frac{a}{3}\,
\begin{array}{|rr|rr|rr|}
 \hline 1  & 1& 1  & 0 & 0 & 0 \\
 0 &  0 &  0&  1& -1& -1\\
 0 &  0& 0& 1&  -1 & -1 \\
 0 &  0& 0&  1& -1 & -1\\ \hline
\end{array},\quad 
\frac{a}{3}\,
\begin{array}{|rr|rr|rr|}
 \hline 1  & 1& 1  & 0 & 0 & 0 \\
 0 &  0 &  0&  -1& 1& -1\\
 0 &  0& 0& -1&  1 & -1 \\
 0 &  0& 0&  -1& 1 & -1\\ \hline
\end{array}
\right\}.
\end{split}
\]
and 
\[
\begin{split}
&P_0^g(X_2)(\frac{1}3) \\ 
= 
& \tiny\left\{ 
\frac{a}{3}\,
\begin{array}{|rr|rr|rr|}
 \hline 4  & -2& -2  & 0 & 0 & 0 \\
 0 &  0 &  0&  0& 0& 0\\
 0 &  0& 0& 0&  0 & 0 \\
 0 &  0& 0&  0& 0 & 0\\ \hline
\end{array}, 
\quad 
\frac{a}{3}\,
\begin{array}{|rr|rr|rr|}
 \hline -2& 4  &  -2  & 0 & 0 & 0 \\
 0 &  0 &  0&  0& 0& 0\\
 0 &  0& 0& 0&  0 & 0 \\
 0 &  0& 0&  0& 0 & 0\\ \hline
\end{array}, 
\quad 
\frac{a}{3}\,
\begin{array}{|rr|rr|rr|}
 \hline -2& -2  & 4  &  0 & 0 & 0 \\
 0 &  0 &  0&  0& 0& 0\\
 0 &  0& 0& 0&  0 & 0 \\
 0 &  0& 0&  0& 0 & 0\\ \hline
\end{array}
\right\}.
\end{split}
\]

Similarly,  $\varphi^3= \sigma_\eta \phi_{g^3}$ for some $\eta\in (\Lambda^{g^3})^*$ with 
$(\eta| \eta) = 1/2$. Set 
\[
X_3=\{ y\in (\eta + (\Lambda^{g^3})^*)(\frac{1}2) \mid P_0^g(y) \in 3\mu + (\Lambda^{g})^*\}.
\]
Then $|X_3|=6$ and $P_0^g(X_3)$ has $6$ vectors of norm $1/6$;
\[
\begin{split}
&P_0^g(X_3)(\frac{1}6) \\ 
= 
& \tiny\left\{ 
\frac{a}{3}\,
\begin{array}{|rr|rr|rr|}
 \hline 0  & 0& 0  & 0 & 0 & 0 \\
 0 &  0 &  0&  \pm 2& 0& 0\\
 0 &  0& 0& \pm 2&  0 & 0 \\
 0 &  0& 0&  \pm 2& 0 & 0\\ \hline
\end{array}, 
\quad 
\frac{a}{3}\,
\begin{array}{|rr|rr|rr|}
 \hline 0  & 0& 0  & 0 & 0 & 0 \\
 0 &  0 &  0&  0&\pm 2&  0\\
 0 &  0& 0& 0 &\pm 2&   0 \\
 0 &  0& 0& 0 & \pm 2&  0\\ \hline
\end{array}, 
\quad 
\frac{a}{3}\,
\begin{array}{|rr|rr|rr|}
 \hline 0  & 0& 0  & 0 & 0 & 0 \\
 0 &  0 &  0&  0& 0& \pm 2\\
 0 &  0& 0& 0 &   0&\pm 2 \\
 0 &  0& 0&  0 &  0&\pm 2\\ \hline
\end{array} 
\right\}.
\end{split}
\]

Notice that the sublattice 
\[
\Lambda^{g, \nu}=\{ x\in \Lambda^g\mid ( x| \nu) \in \Z\}
\]
is isometric to $\sqrt{6}D_4 +\sqrt{2}A_2$. Therefore, $V_\Lambda^\varphi$ contains a full subVOA isomorphic to $V_{\sqrt{6}D_4 +\sqrt{2}A_2} \otimes V_{\Lambda_g}^{\hat{g}}$ and 
$(V_{\sqrt{6}D_4 +\sqrt{2}A_2}, V_{\Lambda_g}^{\hat{g}})$ forms a dual pair in  $V_\Lambda^\varphi$. Therefore, there is an isometry $\psi:  (R(V_{\sqrt{6}D_4 +\sqrt{2}A_2}), q) \to (R(V_{\Lambda_g}^{\hat{g}}), q)$ such that  $q(M)+ q(\psi(M))=0\mod \Z$ for all $M\in R(V_{\sqrt{6}D_4 +\sqrt{2}A_2})$.

Recall that the isometry group of $L_\mathfrak{g}\cong \sqrt{6}D_4 +\sqrt{2}A_2$ is isomorphic to the central product 
\[
O(D_4)* O(A_2) \cong  O^+(4,3) * (2.Sym_3) \cong 2^{1+4}. (Sym_3\times Sym_3\times Sym_3), 
\] 
while the discriminant group of $L_\mathfrak{g}$ has the type $ 2^{4}\cdot 4^{2} \cdot 3^5$ and the isometry group is $ O(2^{4}\cdot 4^{2}) \times O(5,3)$, which is a group of order $2^{21}\cdot  3^7 \cdot 5$.  

We will consider an isometry of $\mathcal{D}(L_\mathfrak{g})$ such that the discriminant group $\mathcal{D}(\sqrt{2}A_2)\cong 2^{2,-} \cdot 3$ to a subgroup of $\mathcal{D}(\sqrt{6}D_4)$ of the same type.  More precisely, let $\zeta: \mathcal{D}(L_\mathfrak{g}) \to \mathcal{D}(L_\mathfrak{g})$ be an isometry such that 
\[ \tiny
\begin{split}
&a\cdot
\begin{array}{|rr|rr|rr|}
 \hline  2  & -2& 0  & 0 & 0 & 0 \\
 0 &  0 &  0&  0& 0&  0\\
 0 &  0& 0& 0 &   0& 0\\
 0 &  0& 0&  0 &  0& 0\\ \hline
\end{array}
+ L_\mathfrak{g}
\longmapsto 
a\cdot
\begin{array}{|rr|rr|rr|}
 \hline 0  & 0& 0  & 0 & 0 & 0 \\
 0 &  0 &  0&  2&  -2& 0\\
 0 &  0& 0&  2 & -2 &   0\\
 0 &  0& 0& 2&  -2 &  0\\ \hline
\end{array} + L_\mathfrak{g}\\
& a\cdot 
\begin{array}{|rr|rr|rr|}
 \hline  0& 2  & -2  & 0 & 0 & 0 \\
 0 &  0 &  0&  0& 0& 0\\
 0 &  0& 0& 0 &   0&0\\
 0 &  0& 0&  0 &  0&0\\ \hline
\end{array}
+ L_\mathfrak{g}
\longmapsto  
a\cdot
\begin{array}{|rr|rr|rr|}
 \hline 0  & 0& 0  & 0 & 0 & 0 \\
 0 &  0 &  0& 0& 2&  -2\\
 0 &  0& 0&    0 &2 & -2 \\
 0 &  0& 0&  0& 2&  -2 \\ \hline
\end{array} + L_\mathfrak{g}\\
&\frac{a}{3}\cdot
\begin{array}{|rr|rr|rr|}
 \hline  2  & -4& 2  & 0 & 0 & 0 \\
 0 &  0 &  0&  0& 0& 0\\
 0 &  0& 0& 0 &   0&0\\
 0 &  0& 0&  0 &  0& 0\\ \hline
\end{array} + L_\mathfrak{g}
\longmapsto 
\frac{a}{3}\cdot
\begin{array}{|rr|rr|rr|}
 \hline 0  & 0& 0  & 0 & 0 & 0 \\
 0 &  0 &  0&  2& 0&  -2\\
 0 &  0& 0&  2 &   0& -2 \\
 0 &  0& 0& 2&   0 &-2 \\ \hline
\end{array}+ L_\mathfrak{g}
\end{split}
\]
and 
\[ \tiny
\begin{split}
\frac{a}{3}\cdot
\begin{array}{|rr|rr|rr|}
 \hline  2  & 2& 2  & 0 & 0 & 0 \\
 0 &  0 &  0&  0& 0&  0\\
 0 &  0& 0& 0 &   0& 0\\
 0 &  0& 0&  0 &  0& 0\\ \hline
\end{array}+ L_\mathfrak{g}, \quad 
\frac{a}{3}\cdot
\begin{array}{|rr|rr|rr|}
 \hline 1  & 1& 1  & 0 & 0 & 0 \\
 0 &  0 &  0&  1&  1& 1\\
 0 &  0& 0&  1 & 1 &  1\\
 0 &  0& 0&  1&  1 &  1\\ \hline
\end{array}+ L_\mathfrak{g},
\quad 
& \frac{a}{3}\cdot
\begin{array}{|rr|rr|rr|}
 \hline  0& 0  & 0  & 0 & 0 & 0 \\
 0 &  0 &  0&  4&  0&4\\
 0 &  0& 0& 4   &0 &4 \\
 0 &  0& 0&  4 & 0 &4\\ \hline
\end{array} + L_\mathfrak{g}
\end{split}
\]
are fixed by $\zeta$. 

Then $\psi\circ \zeta$ defines an isometry from $(R(V_{\sqrt{6}D_4 +\sqrt{2}A_2}), q)$ to $(R(V_{\Lambda_g}^{\hat{g}}), q)$ such that  $q(M)+ q(\psi\circ \zeta(M))=0\mod \Z$. That means $\psi\circ \zeta$ defines another  holomorphic extension $V(\psi\circ \zeta)$ of 
$V_{\sqrt{6}D_4 +\sqrt{2}A_2} \otimes V_{\Lambda_g}^{\hat{g}}$. It is clear that the weight one Lie algebra of $V(\psi\circ \zeta)$ still has Lie rank $6$. By a direct calculation, it is easy to verify that $V(\psi\circ \zeta)_1$ contains a Lie subalgebra of the type $A_{2,12} A_{2,6} A_{2,2}$. Based on Schellekens' list \cite{Sc93,EMS}, the only possibility for $V(\psi\circ \zeta)_1$ is $F_{4,6}A_{2,2}$.

\medskip

\paragraph{\textbf{Acknowledgment}}  The author thanks Hiroki Shimakura for simulating discussion and comments. He also thanks the referee for very helpful comments. A preliminary version of this article has been reported in a publication of RIMS, Kyoto University.

\end{document}